\documentclass[12pt]{amsart}
\usepackage{amsmath,amsfonts,amssymb,mathrsfs,amstext,amscd,latexsym,amsthm}
\usepackage[margin=1.2in]{geometry}
\usepackage{comment}
\usepackage{mathtools}
\usepackage{hyperref}
\usepackage{enumitem}
\usepackage{xcolor}
\usepackage{relsize}
\usepackage{fancyhdr}
\usepackage{bigints}

\def\labelitemi{--}
\def\ba #1\ea {\begin{align} #1\end{align}}
\def\bann #1\eann {\begin{align*} #1\end{align*}}
\def\ben #1\een {\begin{enumerate} #1\end{enumerate}}
\def\bi #1\ei {\begin{itemize}\renewcommand\labelitemi{--} #1\end{itemize}}
\theoremstyle{plain}
\newtheorem{thm}{Theorem}[section]
\newtheorem*{thm*}{Theorem}
\newtheorem{lem}[thm]{Lemma}

\newtheorem{cor}[thm]{Corollary}

\newtheorem{prop}[thm]{Proposition}

\theoremstyle{remark}
\newtheorem{rem}{Remark}
\theoremstyle{problem}
\newtheorem*{problem}{Problem}
\theoremstyle{Yau}
\newtheorem*{Yau}{Yau's conjecture}
\theoremstyle{Chern}
\newtheorem*{Chern}{Chern's conjecture}
\theoremstyle{Chern'}
\newtheorem*{Chern'}{Chern's problem}

\usepackage{cite}

\author[Y.~Zhao]{Yuhang Zhao}
\address{School of Mathematics, Nanjing University, Nanjing, 210093, P. R. of China}
\email{yuhangzhao@smail.nju.edu.cn}

\title{ The first  eigenvalue of embedded minimal hypersurfaces in the unit sphere}

\begin{document}

\pagestyle{fancy}
\fancyhead{}
\fancyhf{}
\fancyhead[LE]{\footnotesize\thepage}
\fancyhead[RO]{\footnotesize\thepage}
\renewcommand{\headrulewidth}{0pt}
\begin{abstract}
		In this article,  we prove  that for an embedded minimal hypersurface $\Sigma^{m}$ in $S^{m+1}$, the first eigenvalue $\lambda_1$ of  the Laplacian operator on $\Sigma$ satisfies:  $$\lambda_1> \frac{m}{2}+G(m, |A|_{\max}, |A|_{\min} ) ,$$ 
        where $|A|_{\max}$ and $|A|_{\min}$ denote the maximum and minimum of the norm of the second fundamental form on $\Sigma$, respectively; $G(m, |A|_{\max}, |A|_{\min} )$ is a positive constant that depends only on $m,|A|_{\max}, |A|_{\min}$. In  particular, when the norm $|A|$  of the second fundamental form is constant, we can obtain a gap depending only on $m$, i.e., 
        $$\lambda_1>\left(\frac{1}{2}+ c \right)m ,$$
        where $c$ is a positive absolute constant. 
        This improves Choi and Wang's previous result \cite{chw1983first} that   $\lambda_1\geq \frac{m}{2}$. Our result  shows that one can improve Choi and Wang's result directly without proving Chern's conjecture. This also generalizes Tang and Yan's work \cite{tangyan2013isoparametric}. 
       Based on the proof of the result above, using  the lower bound of the first Steklov eigenvalue, we  prove that if the  norm $|A|$  of the second fundamental form is constant, then 
        $$|A| \leq \frac{C(m)\textup{Volume}(\Sigma)}{\textup{Volume}(S^m)},$$
        where $C(m)$ is a constant that depends only on $m$. This provides a uniform  estimate for the scalar curvature of  embedded minimal hypersurfaces with  constant norm of the second fundamental form.  Moreover, this may be useful for Chern's problem.
        \end{abstract}

    \maketitle
	
\tableofcontents
	
	\section{Introduction}
	
	Let $F:\Sigma^m \rightarrow S^{m+1}$ be a minimal immersion, where $\Sigma$ is compact. It is  well known that  the restriction of any coordinate function of $\mathbb{R}^{m+2}$ to $\Sigma$  is an eigenfunction of the Laplacian operator of $\Sigma$ with eigenvalue $m$, that is, 
    \begin{align}\label{definition}
\Delta^{\Sigma}F=-m F.
    \end{align}
     This implies that the first eigenvalue (of the Laplacian) of  $\Sigma$ is smaller than or equal to $m$.

In \cite{yau1982seminar}, S.T.Yau raised the following  conjecture:
\begin{Yau}
The first eigenvalue of any  compact  embedded minimal hypersurface in $S^{m+1} $ is $m$.
\end{Yau}
In \cite{chw1983first}, Choi and Wang made the first breakthrough. They used Reilly's formula to get $\lambda_1\geq \frac{m}{2}$. In fact, as observed by Brendle \cite[Theorem ~5.1]{bre2013minimal}, Xu-Chen-Zhang \cite{xu1996spectrum} and Barros \cite{bar1999estimates}, the strict inequality $\lambda_1 >\frac{m}{2}$  holds. For $m=2$, by the compactness result \cite{chsch1985space}, we easily yield  $\lambda_1 \geq 1+\epsilon_{g}$(where $\epsilon_g>0$ is a constant depending only on the genus $g$). For the  special case, Choe and Soret \cite{chs2009first} verified that Yau's conjecture is true for the examples constructed by Lawson \cite{Lawson1970complete}  and Karcher-Pinkall-Sterling \cite{Pinkall1988new}.
Tang and Yan\cite{tangyan2013isoparametric}, on the other hand, proved it  under the assumption that the minimal  hypersurface is isoparametric. 

Recently, the author \cite{zhao2023first} improved $\lambda_1>1$  to $$\lambda_1>1+ G(\cot^{-1} |A|_{\max}) ^6$$ in the two-dimensional case, where $G$ is an absolute positive  constant and $|A|_{\max}$ denotes the maximum of the norm of the second fundamental form. Subsequently, Duncan,  Spruck  and Sire \cite{dun2024improved} generalized the result to higher dimensions and obtained a similar gap. Later, Jiménez, Tapia and Zhou \cite{zhou2024lower} improved it further to $$\lambda_1> \frac{m}{2}+ \frac{m(m+1)}{32(12|A|_{\max}+m+11)^2+8}.$$ However, none of these results gives a gap  depending only on $m$. Moreover, these gaps tend to zero as $|A|_{\max}$ tends to infinity; in fact, there are infinitely many examples (e.g., doublings) showing that $|A|_{\max}$ can be arbitrarily large. Consequently, Choi and Wang's theorem remains the best result concerning Yau's conjecture at present.

In this article, we   obtain a better estimate and prove the following main theorem. Since Yau's conjecture holds  trivially   for  totally geodesic spheres, we  need only  consider the non-totally geodesic case, in which $|A_{\max}\geq\sqrt{m} $ from the work of \cite {Simons1968minimal}.
\begin{thm}[Main theorem]\label{main}
Let $F:\Sigma^m \rightarrow S^{m+1} (m\geq 2)$ be a minimal embedding.  If        $\Sigma$ is not totally geodesic, then the first (nonzero) eigenvalue $\lambda_1$  of   $\Sigma$ (with respect to the induced metric) satisfies: \begin{align}
    \lambda_1>\frac{m}{2}+ \frac{\sqrt{m^2-1}}{48|A|_{\max}} \left[\left(\frac{10}{\sqrt{11}}-\sqrt{\frac{m+1}{m}}\right)|A|_{\min} + \frac{13}{2\sqrt{11}}\sqrt{m}-\frac{m+2}{\sqrt{m+1}} \right].\notag
    \end{align}
where
$|A|_{\max}$ and $|A|_{\min}$ denote the maximum and minimum of the norm of the second fundamental form, respectively.
\end{thm}
One can  see directly that even scaling $|A|_{\min}$ to zero yields a better result than the recent results \cite{zhao2023first, dun2024improved, zhou2024lower}. Moreover,  when  the ratio $\frac{\max\{|A|_{\min},\sqrt{m} \}}{|A|_{\max}}$ has a universal lower bound,  we can get 
\begin{align}
\lambda_1>\left(\frac{1}{2}+c\right)m,\label{form}
\end{align}
 where $c$ is a  positive absolute  constant.  For this case, a  corollary is  when the norm  of the second fundamental form  $|A|$ is constant. We list it below separately because of its significance.
\begin{thm}\label{main'}
Let $F:\Sigma^m \rightarrow S^{m+1} (m\geq 2)$ be a minimal embedding. 
If the norm  of the second fundamental form  is constant, then the first (nonzero) eigenvalue $\lambda_1$ of   $\Sigma$ (with respect to the induced metric) satisfies: 
\begin{align}
    \lambda_1>\frac{m}{2}+ \frac{\sqrt{m^2-1}}{48}\left(\frac{10}{\sqrt{11}}-\sqrt{\frac{m+1}{m}}\right). \notag
    \end{align}
\end{thm}
One can observe that this result satisfies  \eqref{form}.  More precisely, as $m$
tends to infinity, the gap is asymptotic to $\frac{1}{48}\left(\frac{10}{\sqrt{11}} -1 \right)m \approx 0.042 m$. Furthermore, we believe that this result can be further improved, as  some estimates are rough in the proof. This paper  aims to yield an explicit gap depending only on $m$. There is an interesting problem here:
\begin{problem}
Is it possible to improve this gap to a value  close to $\frac{m}{2}$? 
\end{problem}
Regarding the case where the norm of the second fundamental form is constant, much progress has been made in recent decades, and  the most significant is Chern's conjecture:
\begin{Chern}
  Let $ F :\Sigma^m \rightarrow S^{m+1}$ be a minimal immersion. If the  norm of   the second fundamental form is constant, then $\Sigma$ is isoparametric.
\end{Chern}
The conjecture was originally proposed in a less  strong version by Chern in \cite{Chernminimal} and \cite{CherndoCarmominimal}. So far,  progress on Chern's conjecture has only been made completely in dimensions 2 and 3, and  partially in  higher dimensions. The latest advance is a recent result proved by He, Xu and Zhao \cite{XuChern}, which states that any closed minimal hypersurface $\Sigma^4$ in $S^5$ with constant scalar curvature and constant 3-th mean curvature must be isoparametric. For other related progress, see \cite{CherndoCarmominimal, Lawson1969rigidity, Pengminimal, Changminimal, Lusalawillmore2005, Dengwillmore }.

Concerning  the relationship between Chern's conjecture and Yau's conjecture,  Tang and Yan \cite{tangyan2013isoparametric} proved that if Chern's conjecture holds, then Yau's conjecture will   also hold under the assumption that the norm of the second fundamental  form is  constant. However, our result (see Theorem \ref{main'} )  shows that  one can  skip proving Chern's conjecture to improve Choi and Wang's result \cite{chw1983first} directly. Furthermore, our proof is  completely different from  theirs, and it is suitable for general embedded minimal hypersurfaces.  Compared to their complete resolution  in the isoparametric case, we can only obtain  a very small gap. This is because, in the case that the norm of the  second fundamental form is constant, we know very little about the structure of the minimal hypersurface , especially in dimensions greater than four. Moreover, unlike their work, which depends deeply on the structure and  classification of isoparametric minimal hypersurfaces, our  proof uses only basic information about minimal hypersurfaces.

To prove the  main theorem, we need to prove the following theorem:
\begin{thm}\label{computation}
Let $F:\Sigma^m \rightarrow S^{m+1} (m\geq 2)$ be a minimal embedding, the image $F(\Sigma)$  of which divides  $S^{m+1}$ into two connected regions: $\Omega_1$ and $\Omega_2$ such that $\partial \Omega_1=\partial \Omega_2=\Sigma$. Identify $\Sigma$ with $F(\Sigma)$. Let $u$ and $v$ be smooth up to the boundary on $\Omega_1$ and $\Omega_2$, respectively, such that $u|_{\Sigma}=v|_{\Sigma}=f$, where $f$ is a smooth function on $\Sigma$. Then 
\begin{align}
& 2\int_{\Sigma} A\big(\nabla^{\Sigma}(\Delta v-\Delta u),\nabla^{\Sigma} f\big)+2 \int_{\Sigma} \langle \nabla u,\mathbf{n} \rangle \cdot \Delta^{\Sigma} (\Delta u-2\Delta^{\Sigma}f)\notag\\
&-2 \int_{\Sigma} \langle \nabla v,\mathbf{n} \rangle \cdot \Delta^{\Sigma} (\Delta v-2\Delta^{\Sigma}f)+4m \int_{\Sigma} \big( \langle \nabla v,\mathbf{n} \rangle - \langle \nabla u,\mathbf{n} \rangle \big)\cdot\Delta^{\Sigma} f\notag\\
&+5\int_{\Sigma}A\Big(\big( D_{\mathbf{n}}\nabla v\big)^{\top},\big( D_{\mathbf{n}}\nabla v\big)^{\top}\Big)- 5\int_{\Sigma}A\Big(\big(D_{\mathbf{n}}\nabla u\big)^{\top},\big( D_{\mathbf{n}}\nabla u\big)^{\top}\Big) \notag\\
& +2\int_{\Sigma} (\Delta u-\Delta^{\Sigma} f )\cdot \Big\langle A,F^{\star}\big(D^2u \big)\Big\rangle -2\int_{\Sigma} (\Delta v-\Delta^{\Sigma} f )\cdot \Big\langle A,F^{\star}\big(D^2v \big)\Big\rangle\notag\\
&+ \int_{\Sigma} \textup{Trace}_{\Sigma} \Bigg (A\Big(\big( D_{()} \nabla v \big)^{\top} ,\big( D_{()} \nabla v \big)^{\top}\Big)\Bigg)\notag\\
&-\int_{\Sigma} \textup{Trace}_{\Sigma} \Bigg (A\Big(\big( D_{()} \nabla u \big)^{\top} ,\big( D_{()} \nabla u \big)^{\top}\Big)\Bigg)\notag\\
=&\int_{\Omega_1}|D^3 u|^2 + \int_{\Omega_2}|D^3 v|^2 -2m(m+1)\int_{\Omega_1}|\nabla u|^2- 2m(m+1)\int_{\Omega_2}|\nabla v|^2\notag\\
&-\int_{\Omega_1}|\nabla \Delta u |^2-\int_{\Omega_2}|\nabla \Delta v |^2
+m\int_{\Omega_1}(\Delta u)^2 +m\int_{\Omega_2}(\Delta v)^2 \notag\\
&-2m\int_{\Omega_1}\langle\nabla \Delta u,\nabla u \rangle-2m\int_{\Omega_2}\langle\nabla \Delta v,\nabla v    \rangle,\notag
\end{align}
where  $\mathbf{n}$ is the  inward-pointing  unit normal vector field on $\Sigma$ with respect to $\Omega_1$,  $A$ is the corresponding  second fundamental form,  $\top$ denotes the projection onto the tangent bundle of $\Sigma$, $F^{\star}$ denotes the pull-back of tensors corresponding to the map: $F:\Sigma \rightarrow S^{m+1} $ and  $$\textup{Trace}_{\Sigma} \Bigg (A\Big(\big( D_{()} \nabla u \big)^{\top} ,\big( D_{()} \nabla u \big)^{\top}\Big)\Bigg)=\sum_{i=1}^{m}
A\Big(\big( D_{\overline{e_i}} \nabla u \big)^{\top} ,\big(D_{\overline{e_i}} \nabla u \big)^{\top}\Big)$$ in any local orthonormal frame $\{\overline{e_l} \}_{l=1}^{m}$ on $\Sigma$ \textup{(}here $\overline{e_i}$ is identified with its image under the tangent map with respect to $F$\textup{)}. Other specific notations are defined in Section \ref{preliminaries}.
\end{thm}
  Here we compute the general case, although in the proof of the main theorem, we choose $f$ to be an eigenfunction with eigenvalue $\lambda_1$ and let  $u$ and $v$   be harmonic extensions of $f$ to the interior of  the regions $ \Omega_1$ and $\Omega_2$, respectively. The approach  to this theorem is to directly compute the Laplacians of $|D^2 u|^2$ and $|D^2 v|^2$,  then perform integration by parts on them over the regions $\Omega_1$ and $\Omega_2$, respectively,  and finally add the two integrals to cancel some boundary terms.  The main motivation for proving this theorem is that, in Choi and Wang's  proof \cite{chw1983first}, they only computed the Laplacians of $|\nabla u|^2$ and $|\nabla v|^2$ and the term $$\int_{\Omega_1}|D^2 u|^2+\int_{\Omega_2} |D^2 v|^2 $$  was be simply thrown away, which we consider insufficient. Thus, it is necessary to compute higher-order derivatives to analyze the Hessian term carefully. Considering the particular setting of $S^{m+1}$, we  adopt a  slightly simpler and clearer approach than the general computation of derivative of  tensors, the  advantage of which  is that we can directly use Reilly's  formula to compute the integrals $\int_{\Omega_1}\Delta |D^2 u|^2$ and  $\int_{\Omega_2}\Delta |D^2 v|^2$.    The computational details can be found in Section \ref{computation'}.

  To prove the main theorem, we first need the  fact that the rolling radii $d_1$ and $ d_2$ of  $\Omega_1$  and $\Omega_2$ are equal to their respective focal distances, see Theorem \ref{rolling}. Let $d_0=\min\{ d_1, d_2\}$.  Hence, $d_0$ has a lower bound depending on the maximum of the  norm of the second fundamental form on $\Sigma$, and  the exponential map $\exp_{p}(r\mathbf{n}(p))$ defines a diffeomorphism between $(-d_0, d_0) \times \Sigma $ and the tubular neighborhood $ \{p\in S^{m+1}| \textup{distance}(p, \Sigma)<d_0 \} $. Based on these facts, using  the cut-off function on this tubular neighborhood, we prove that when $\Sigma$ is not totally geodesic,
\begin{align}
    &\int_{\Sigma}(|D^2 u|^2 +|D^2 v|^2) \notag\\
    \leq &\mathcal{K}(|A|_{\max},m) \Bigg( \int_{\Omega_1} |D^2 u|^2 + \int_{\Omega_2} |D^2 v|^2\Bigg)+ \frac{\mathcal{E}}{|A|_{\max} } \Bigg(\int_{\Omega_1} |D^3 u|^2 +\int_{\Omega_2} |D^3 v|^2\Bigg).\notag
    \end{align}
where $\mathcal{K}(|A|_{\max},m) $  is a positive constant depending only on $m$ and $|A|_{\max}$, and $\mathcal{E}$ is a sufficiently small absolute constant. 

Next, we combine this estimate, Theorem \ref{computation} and Choi and  Wang's work, substitute $f$ as the eigenfunction with eigenvalue $\lambda_1$, let  $u$ and $v$ be the corresponding harmonic extensions, and then bound the integrals of the boundary terms in Theorem \ref{computation}.  Then  the proof follows from the expression of the term $$ |D^2 u|^2 +|D^2 v|^2 $$ restricted to the boundary $\Sigma$ (see \eqref{boundary information 3}). The  detailed steps are provided in Section \ref{main proof}. 

In fact, in the proof of the main theorem, we can  find that when the norm $|A|$ of the second fundamental form is constant and $\Sigma$ is not totally geodesic, for the functions $f,u,v$ above, we have
$$\frac{\int_{\Omega_1}|\nabla u |^2 +\int_{\Omega_2}|\nabla v|^2  }{\int_{\Sigma} f^2 }<\frac{E(m)}{|A|},$$ where $E(m)$ is a constant that depends only on $m$. The term $$\frac{\int_{\Omega_1}|\nabla u |^2 +\int_{\Omega_2}|\nabla v|^2  }{\int_{\Sigma} f^2 }$$ is directly related to the lower bound of the first nonzero eigenvalue of the Dirichlet- Neumann map. In other words, once the  lower bound of the first nonzero spectrum of the Dirichlet-to-Neumann map is obtained, then  together with the upper bound above,  we arrive at the following theorem, see Section \ref{curvature-steklov}. The discussion of the Dirichlet- Neumann map is presented after Corollary \ref{curvature upper bound'}.
\begin{thm}\label{curvature upper bound}
Let $F:\Sigma^m \rightarrow S^{m+1} (m\geq 2)$ be a minimal embedding. 
If the norm  of the second fundamental form  $|A|$ is constant, then $$|A| \leq  \frac{C(m)\textup{Volume}(\Sigma)}{\textup{Volume}(S^m)},$$
where $$C(m)=\frac{ 9\sqrt{m(m-1)}+\frac{8(m+1)\sqrt{m}}{5\sqrt{m-1}} }{1-\frac{4(2m^2-3m+2)}{25(m-1)^2}}\left(\frac{1}{\sin(\delta_m)}+(m+1)\delta_m \right), $$
$$\sin^2(\delta_m)= \frac{2}{\sqrt{4(m+1)^2+1}+1}\ ,0<\delta_m<\frac{\pi}{2}.$$
\end{thm}
The result yields a uniform estimate for embedded  minimal hypersurfaces with constant norm of the second fundamental form. Moreover, it may provide some evidence for the following Chern's problem, which is stated as follows:
\begin{Chern'}
Let $F :\Sigma^m \rightarrow S^{m+1}$ be a minimal immersion. If the norm  of the  the second fundamental form is constant, then does  there exist a positive constant $\mathcal{C}(m)$ that  depends only on  $m$   such that $$|A| \leq \mathcal{C}(m)?$$
\end{Chern'}
It is a slightly weaker version of Chern's conjecture and was proposed  by Chern in \cite{CherndoCarmominimal}. From theorem \ref{curvature upper bound}, under the additional  embedding assumption, if one  can prove that the volume of $\Sigma$ admits  a uniform upper bound, then Chern's problem is solved. Volume is often a better quantity than curvature in differential geometry. 
Combining Theorem \ref{curvature upper bound}, Theorem \ref{main'} and Corollary \ref{area curvature}, we immediately obtain the following corollaries:
\begin{cor}\label{eigenvalue'}
    Let $F:\Sigma^m \rightarrow S^{m+1} (m\geq 2)$ be a minimal embedding. 
If the norm  of the second fundamental form  is constant, then the first (nonzero) eigenvalue $\lambda_1$ of   $\Sigma$ (with respect to the induced metric) satisfies: 
\begin{align}
    \lambda_1>\frac{m}{2}+ \frac{\sqrt{m^2-1}}{48}\left[\frac{10}{\sqrt{11}}-\sqrt{\frac{m+1}{m}} + \frac{\frac{13}{2\sqrt{11}}\sqrt{m}-\frac{m+2}{\sqrt{m+1}}}{C(m)}\cdot \frac{\textup{Volume}(S^m)}{\textup{Volume}(\Sigma)}\right], \notag
    \end{align}
where $C(m)$ is the constant in Theorem \ref{curvature upper bound}\end{cor}
\begin{cor}\label{curvature upper bound'}
Let $F:\Sigma^m \rightarrow S^{m+1} (m\geq 2)$ be a minimal embedding. If $\Sigma$ is not totally geodesic and  the norm $|A|$ of the second fundamental form is constant,
    then 
    $$ \sqrt{\frac{m-1}{m}}\frac{\textup{Volume}(\Sigma)}{\textup{Volume}(S^{m+1})} \leq |A| \leq  \frac{C(m)\textup{Volume}(\Sigma)}{\textup{Volume}(S^m)}\ ,$$
\end{cor}
where $C(m)$ is the constant in Theorem \ref{curvature upper bound}.

To prove Theorem \ref{curvature upper bound}, we recall the definition of the Dirichlet-to-Neumann map.
For the  region $\Omega_1 $,  the Dirichlet-to-Neumann  map $\wedge_1: C^{\infty}(\Sigma)\rightarrow C^{\infty}(\Sigma)$ is defined by
$$ \wedge_1 g= -\langle \nabla(\mathcal{H}_1g ),\mathbf{n}\rangle ,$$ where  $\mathbf{n}$ is the inward-pointing unit normal vector field and $\mathcal{H}_1 g$
is the harmonic extension of $g$ to the interior of $\Omega_1$.
For the region $\Omega_2$, we denote this map by $\wedge_2$. Then 
$$\wedge_2 g=\langle \nabla(\mathcal{H}_2g ),\mathbf{n}\rangle. $$
The Steklov eigenvalues constitute  the  spectrum of the Dirichlet-to-Neumann map. Here, considering that in the proof of the main theorem, there are two domains and two harmonic functions that are identical on $\Sigma$, we need the the following map:
$$\wedge =\wedge_1 +\wedge_2.$$
For convenience, we still refer to it  as the Dirichlet-to-Neumann map, and the corresponding 
eigenvalues are still called the Steklov eigenvalues. 
 The map $\wedge$ also appears in \cite[Example~2.22]{Girouard2024}.

 A standard variational principle for the first nonzero Steklov eigenvalue (of $\wedge$)  is given by 
 $$\tau_1=\inf\limits_{g\in C^1 (\Sigma), \ \int_{\Sigma} g =0 } \frac{\int_{\Omega_1}|\nabla (\mathcal{H}_1 g)|^2+\int_{\Omega_2}|\nabla (\mathcal{H}_2 g)|^2}{\int_{\Sigma} g^2 }.$$
From \eqref{definition}, we know that the integral of each coordinate function of $\Sigma$ is zero. Then, applying the variation characterization for $\tau_1$ to each coordinate function of $\Sigma$ yields directly 
\begin{align}
  \tau_1 \leq   \frac{(m+1)\textup{Volume}(S^{m+1})}{\textup{Volume}(\Sigma)}.\label{steklov upper bound}
\end{align}
With these preparations, we have the following theorem , see also  Theorem \ref{steklov'}.
\begin{thm}\label{steklov}
Let $F:\Sigma^m \rightarrow S^{m+1} (m\geq 2)$ be a minimal embedding. Then the first nonzero Steklov eigenvalue $\tau_1$ ( of the map $\wedge$ )  satisfies
\begin{align}
\tau_1 \geq \frac{\textup{Volume}(S^m)}{D(m)\textup{Volume}(\Sigma)}\ ,  \notag  
\end{align}
where $$D(m)=\frac{1}{\sin(\delta_m)}+(m+1)\delta_m \ ,\ \sin^2(\delta_m)= \frac{2}{\sqrt{4(m+1)^2+1}+1} \ , \ 0<\delta_m<\frac{\pi}{2}.$$
\end{thm}
 \begin{rem}
  Under the particular setting of spheres, only the volume growth of minimal hypersurfaces is used in our proof.
\end{rem}
Combining this theorem and \eqref{steklov upper bound}, we can obtain the upper and lower bounds for the  first Steklov eigenvalue that depend only on $m$ and the volume of $\Sigma$.

The main idea of our proof is to glue together the two harmonic extensions corresponding to the eigenfunction with eigenvalue $\tau_1$ to form a globally Lipschitz function on $S^{m+1}$, and then apply the mean value formula on spheres to  this function. Moreover, different from the  conventional  proof, since our function is not globally smooth,  we must  also deal with  the integral over $\Sigma $. Finally, we integrate this formula over $\Sigma$ and use the volume growth of $\Sigma$ (see Proposition \ref{volume growth} ) to complete the proof. The related details can be found in Section \ref{curvature-steklov}. 
Here, we also have an interesting problem:
\begin{problem}
  Under the assumption of Theorem \ref{steklov}, denote the first Steklov eigenvalues of $\wedge_1$ and $\wedge_2$  by $\tau_1(\Omega_1) $ and $\tau_1 (\Omega_2)$, respectively. From the definition, we know $$\tau_1\geq\tau_1(\Omega_1)+\tau_1 (\Omega_2).$$  Then,  is there  a positive constant $\mathcal{B}$ depending only  on $m$  and $\textup{Volume}(\Sigma)$ such that 
  $$ \min\{\tau_1(\Omega_1),\tau_1(\Omega_2)\} \geq\mathcal{B}\tau_1?$$
\end{problem}
\begin{rem}
    If we relax $\mathcal{B}$ so  that it can depend on $|A|_{\max}$, then this conclusion follows from \cite{ColboisGirouardsteklov2020}.
\end{rem}

The paper is organized as follows. In Section \ref{preliminaries}, we recall the definition of the tubular neighborhood and use it to prove Proposition \ref{function in tubular}. Then, we review the proof of the volume growth of minimal hypersurfaces in spheres; In Sections \ref{computation'}, \ref{main proof} and \ref{curvature-steklov},  we present the proofs of Theorems \ref{computation}, \ref{main}, \ref{curvature upper bound}, \ref{steklov}.

\section{Preliminaries}\label{preliminaries}
Let  $F:\Sigma^{m}  \rightarrow S^{m+1} (m\geq 2) $ be a compact embedding,  where $$S^{m+1}=\{(x_1,\cdots, x_{m+2}) \in \mathbb{R}^{m+2} | x_1^2 +\cdots+ x_{m+2}^2=1   \}.$$
 Let $\langle, \rangle$ and $\cdot$ be the standard Euclidean metric and  dot product, respectively;  let $D,\nabla$ and $\Delta$ be the Levi-Civita connection, gradient and Laplacian on $S^{m+1}$, respectively; and let $D^{\Sigma},\nabla^{\Sigma}$and $\Delta^{\Sigma} $ (with respect to the induced metric) be the Levi-Civita connection, gradient and Laplacian on $\Sigma$, respectively.  The norm of tensors is denoted by $| \cdot |$ and the inner product of tensors is still denoted by $\langle , \rangle$.

 Identify $\Sigma $ with its image $F(\Sigma)$. According to differential topology, $F(\Sigma) $ divides $S^{m+1} $ into two connected regions $\Omega_1 $ and $\Omega_2$  such that $\Sigma =\partial \Omega_1 =\partial \Omega_2$. For the region  $\Omega_1$ in  $S^{m+1}$, we denote by  $\mathbf n$  the inward-pointing  unit normal vector field on $\Sigma$. The corresponding 
\textbf{second fundamental form} is defined by  
$$A(\eta_1,\eta_2) =\langle D_{\eta_1} \eta_2,\mathbf{n} \rangle ,\eta_1,\eta_2 \in \Gamma(T\Sigma) ,$$ where $\Gamma(T\Sigma)$ is the set of all smooth vector fields on  $\Sigma$. The \textbf{shape operator}  is given by $B(\eta)=- D_{\eta} \mathbf{n} $ for $\eta \in \Gamma(T\Sigma)$. The \textbf{principal curvatures} are the eigenvalues $\mu_1 \geq \mu_2 \geq \cdots \geq \mu_{m} $ of the operator $B$. Since $\Omega_1$ is a region in the unit sphere, the \textbf{focal points}  of $\Sigma$ are given by 
$$\pm\left[\cos (\cot^{-1} \mu_1 )F +\sin (\cot^{-1} \mu_1 )\ \mathbf{n}\right],\cdots , \pm[\cos (\cot^{-1} \mu_m )F +\sin (\cot^{-1} \mu_m )\ \mathbf{n}],$$ and the corresponding \textbf{focal distance} is $\cot^{-1} \mu_1$.\footnote{$\cot^{-1} \mu_1$ is $\arctan \frac{1}{ \mu_1}$ for $\mu_1 >0$, $\frac{\pi}{2}$ for $\mu_1 =0$ and $\pi+\arctan \frac{1}{ \mu_1}$ for $\mu_1 <0$.}
We use $ \frac{\sum\limits_{i=1}^{m} \mu_i }{m} $ for the \textbf{mean curvature } $H$ and $\sum\limits_{i=1} ^{m} \mu_i ^2$ for the norm square $|A|^2$ of the second fundamental form. Since the case of $\Omega_2$ differs from that of $\Omega_1$  only by  a sign, it suffices to  discuss  $\Omega_1$ in the rest of this section.
\bigskip

Next, we only consider the case that $F: \Sigma^{m}  \rightarrow S^{m+1} $ is minimal $(i.e.,H=0 )$, in which 
\begin{align}
  \Delta^{\Sigma}F =-m F.\notag
\end{align}

First, we state  the following rolling theorem, which was proven by Howard \cite[Theorem~3]{Howard1999rolling}. The  theorem plays a key role  in  later eigenvalue  estimates. Here we only state the  special case of \cite{Howard1999rolling}. 
\begin{thm}[\text{\cite{Howard1999rolling}}]\label{rolling}
The (rolling or normal injectivity) radius $d_0$  of   $\Omega_1$ is $\min\limits_{\Sigma}\cot^{-1} \mu_1$ $\in (0,\frac{\pi}{2}]$, where $\mu_1$ is the largest  principal curvature  of $\Sigma$ with respect to $\mathbf{n}$ and $\mu_1$ is a continuous function on $\Sigma$ (hence it can attain a  maximum).
\end{thm}
 The theorem  means that the  map 
\begin{align}
 (p,r) \rightarrow \exp^{\perp} (r \mathbf{n} ) =\cos r F (p) +\sin r \ \mathbf{n} (p)\notag
\end{align} 
is a diffeomorphism from $ \Sigma \times [0,d_0)$ to $\exp^{\perp} \Big( \Sigma \times [0,d_0)\Big)\subset \Omega_1$.

 Under this map, the volume form  of the tubular neighborhood can be, up to a sign with the standard volume form of $S^{m+1}$, written as
 \begin{align}
\prod\limits_{i=1} ^{m}(\cos r-\mu_i \sin r) dr \wedge d\sigma \notag,
 \end{align}
where $d\sigma$ is the volume element of $\Sigma$.

In addition, when a function is restricted to the  tubular neighborhood, it can be regarded as a function of $r$ and $\Sigma$.

So  we have the following proposition:
\begin{prop} \label{function in tubular}
If $\phi$ is a smooth nonnegative function on $\Omega_1$, then 
\begin{align}
\int_{\Sigma} \phi \ d\sigma \leq 2\sqrt{\frac{m}{m-1}}\cdot\max\{|A|_{\max}, \sqrt{m} \}\int_{\Omega_1} \phi +\int_{\Omega_1} |\nabla \phi|\notag,
\end{align}
where $|A|_{\max}=\max\limits_{\Sigma} |A|$.
\end{prop}

\begin{proof}
First, fix $0 < \tilde{d_0} \leq d_0  $ and $k\geq 1$ such that $\tilde{d_0}< \frac{\pi}{2}$. For notational convenience, we denote $\prod\limits_{i=1} ^{m}(\cos r-\mu_i \sin r) \ d\sigma $ by $d\sigma_r$. 
Then 
\begin{align}
\int_{\Sigma} \phi \ d\sigma =& -\int_{0}^{\tilde{d_0}} \left( \int_{\Sigma}\phi (\cos r-\cot \tilde{d_0} \sin r)^{k} d\sigma_r \right)' dr \notag\\
=&\int_{0}^{\tilde{d_0}} dr\int_{\Sigma} (\cos r-\cot \tilde{d_0} \sin r)^{k}
\left(\phi \sum\limits_{i=1}^{m} \frac{\sin r +\mu_i \cos r}{\cos r-\mu_i \sin r}-\phi_{r}\right)  d\sigma_r \notag \\
& +k \int_{0}^{\tilde{d_0}} dr\int_{\Sigma} \phi (\cos r-\cot \tilde{d_0} \sin r)^{k-1}(\sin r + \cot\tilde{d_0} \cos r )d\sigma_r \notag\\
\leq &\int_{0}^{\tilde{d_0}} dr\int_{\Sigma} (\cos r-\cot \tilde{d_0} \sin r)^{k}
\left(\phi \sum\limits_{i=1}^{m} \frac{\sin r +\mu_i \cos r}{\cos r-\mu_i \sin r}+|\nabla \phi| \right)  d\sigma_r \notag \\
& +k \int_{0}^{\tilde{d_0}} dr\int_{\Sigma} \phi (\cos r-\cot \tilde{d_0} \sin r)^{k-1}(\sin r + \cot\tilde{d_0} \cos r )d\sigma_r\notag\\
\leq & \int_{\Omega_1}|\nabla \phi|+ \int_{0}^{\tilde{d_0}} dr\int_{\Sigma} \phi(\cos r-\cot \tilde{d_0} \sin r)^{k}
 \sum\limits_{i=1}^{m} \frac{\sin r +\mu_i \cos r}{\cos r-\mu_i \sin r}   d\sigma_r \notag \\
&+ k \int_{0}^{\tilde{d_0}} dr\int_{\Sigma} \phi (\cos r-\cot \tilde{d_0} \sin r)^{k-1}(\sin r + \cot\tilde{d_0} \cos r )d\sigma_r.
\label{first estimate phi}
\end{align}
Since $\sum\limits_{i=1}^{m} \mu_i =0,\sum\limits_{i=1}^{m} \mu_i^2=|A|^2$ and $\mu_i \leq \cot d_0 <\cot\tilde{d_0}$,
\begin{align}
     \sum\limits_{i=1}^{m}  \frac{\sin r +\mu_i \cos r}{\cos r-\mu_i \sin r}
     &=\sum\limits_{i=1}^{m} \mu_i +\sum\limits_{i=1}^{m}\frac{(1+\mu_i ^2) \sin r}{\cos r-\mu_i \sin r}=\sum\limits_{i=1}^{m}\frac{(1+\mu_i ^2) \sin r}{\cos r-\mu_i \sin r}\notag\\
     & \leq \frac{(|A|^2+m)\sin r}{\cos r-\cot{\tilde{d_0}} \sin r}. \label{second estimate phi}
     \end{align}    
     Substituting  $\eqref{second estimate phi}$ into \eqref{first estimate phi} gives
\begin{align}
&\int_{\Sigma} \phi \  d\sigma \notag\\
\leq& \int_{0}^{\tilde{d_0}} dr\int_{\Sigma} \phi (\cos r-\cot \tilde{d_0} \sin r)^{k-1}
\left[k(\sin r + \cot\tilde{d_0} \cos r )+(|A|^2+m)\sin r \right] d\sigma_r\notag\\
& +\int_{\Omega_1}|\nabla \phi| \notag\\
\leq & \int_{0}^{\tilde{d_0}} dr\int_{\Sigma} \phi (\cos r-\cot \tilde{d_0} \sin r)^{k-1}
\left[k \cot\tilde{d_0} \cos r +(|A|_{\max}^2+m+k)\sin r \right] d\sigma_r\notag\\
& +\int_{\Omega_1}|\nabla \phi|.\label{third estimate phi}
\end{align}
Let
$$ \varphi (r)=(\cos r-\cot \tilde{d_0} \sin r)^{k-1}
\left[k \cot\tilde{d_0} \cos r +(|A_{\max} ^2+m+k)\sin r \right].$$
We need to estimate the maximum of $\varphi$ on the interval $\big[0,\tilde{d_0}\big]$.

 Taking  the derivative of $\varphi$, we  have 
\begin{align}
  \varphi '= & -(k-1) (\cos r-\cot \tilde{d_0} \sin r)^{k-2}(\sin r +\cot\tilde{d_0} \cos r)\notag\\
  &\times\left[k \cot\tilde{d_0} \cos r +(|A_{\max} ^2+m+k)\sin r \right] \notag \\
  &+(\cos r-\cot \tilde{d_0} \sin r)^{k-1}\left[(|A|_{\max}^2+m+k)\cos r-k \cot \tilde{d_0} \sin r\right]\notag \\
  =& (\cos r-\cot \tilde{d_0} \sin r)^{k-2}\cos ^2 r \bigg\{|A|_{\max}^2+m+k-k(k-1) \cot^2 \tilde{d_0}\notag \\
  & +\left[k \cot^2\tilde{d_0} -(k-1)(|A|_{\max}^2 +m+k) \right] \tan^2 r\notag\\
 & -k\cot\tilde{d_0}(|A|_{\max}^2 +m+2k )\tan r\bigg\}.\notag
  \end{align}
We can see that when  choosing  $k$ such that 
$$ k(k-1) \cot^2\tilde{d_0} =|A|_{\max}^2+m+k ,$$
i.e.,
$$k=\frac{\tan^2 \tilde{d_0}+1}{2}+\sqrt{\frac{(\tan^2 \tilde{d_0}+1)^2 }{4}+\tan ^2\tilde{d_0}(|A|_{\max}^2+m) } \ > 1, $$
we have 
\begin{align}
    \varphi'= &(\cos r-\cot \tilde{d_0} \sin r)^{k-2}\cos ^2 r\tan r \cdot \frac{|A|_{\max}^2+m+k}{k-1}
    \notag\\
    &\times \left[(1-(k-1)^2)\tan r-\tan \tilde{d_0}(|A|_{\max}^2 +m+2k )\right]\notag  \\
    \leq &-(\cos r-\cot \tilde{d_0} \sin r)^{k-2}\cos ^2 r\tan^2 r \cdot \frac{|A|_{\max}^2+m+k}{k-1} (k^2+|A|_{\max}^2+m)\leq 0,\notag
    \end{align}
where we use $-\tan \tilde{d_0} \leq -\tan r$.

Hence, $\varphi$ is non-increasing on $\big[0,\tilde{d_0}\big]$ and 
$$\varphi_{\max}=\varphi (0)=k\cot\tilde{d_0}=
\frac{\tan \tilde{d_0}+\cot\tilde{d_0}}{2}+\sqrt{\frac{(\tan \tilde{d_0}+\cot\tilde{d_0})^2 }{4}+|A|_{\max}^2+m }\ .$$

 Plugging this result into \eqref{third estimate phi} yields
\begin{align}
    \int_{\Sigma} \phi \ d\sigma &\leq \int_{\Omega_1}|\nabla \phi| +
\left(\frac{\tan \tilde{d_0}+\cot\tilde{d_0}}{2}+\sqrt{\frac{(\tan \tilde{d_0}+\cot\tilde{d_0})^2 }{4}+|A|_{\max}^2+m } \ \  \right )\int_{\Omega_1}\phi.\label{forth estimate phi}
    \end{align}
When $\Sigma$ is totally geodesic, we let  $\tilde{d_0}= \frac{\pi}{4}$.
Then 
\begin{align}
    \int_{\Sigma} \phi \ d\sigma &\leq \int_{\Omega_1}|\nabla \phi| +
(1+\sqrt{m+1})\int_{\Omega_1}\phi\ .\label{final estimate 1 phi}
\end{align}

When $\Sigma$ is not totally geodesic, by the work  of \cite{Simons1968minimal}, we have $|A|_{\max}^2 \geq m$. 

Since $$ |A|^2=\sum\limits_{i=1}^{m} \mu_i^2 \geq \mu_1^2 +\frac{\left(\sum\limits_{i=2}^{m} \mu_i\right)^2}{m-1}=\frac{m}{m-1}\mu_1^2,$$

$$\cot d_0=\max\limits_{\Sigma} \mu_1 \leq \sqrt{\frac{m-1}{m}} |A|_{\max}.$$
Let 
$$\tilde{d_0} =\cot^{-1} \left(\sqrt{\frac{m-1}{m}} |A|_{\max}\right) \left(\leq \min\{d_0,\cot^{-1}\sqrt{m-1}\} \leq \min\{d_0,\frac{\pi}{4}\}\right).$$
Then  by \eqref{forth estimate phi},
\begin{align}
\int_{\Sigma} \phi \ d\sigma \leq &\int_{\Omega_1}|\nabla \phi| +\left(\frac{1}{2}\sqrt{\frac{m-1}{m}}|A|_{\max}+\frac{1}{2}\sqrt{\frac{m}{m-1}}\frac{1}{|A|_{\max}} \notag \right.\\
&\left.+\sqrt{ \frac{5m-1}{4m}|A|_{\max}^2+m+\frac{1}{2} +\frac{1}{4} \cdot \frac{m}{m-1} \frac{1}{|A|_{\max}^2}    } \ 
\right)\int_{\Omega_1}\phi\ .\label{final estimate 2 phi}
\end{align}
Since $|A|_{\max} \geq \sqrt{m}$, 
\begin{align}
  &\frac{1}{2}\sqrt{\frac{m-1}{m}}\cdot|A|_{\max}+\frac{1}{2}\sqrt{\frac{m}{m-1}}\cdot\frac{1}{|A|_{\max}}\notag\\
&+\sqrt{ \frac{5m-1}{4m}|A|_{\max}^2+m+\frac{1}{2} + \frac{m}{4(m-1)} \frac{1}{|A|_{\max}^2}} \notag\\
=& |A|_{\max}\Bigg (\frac{1}{2}\sqrt{\frac{m-1}{m}}+\frac{1}{2}\sqrt{\frac{m}{m-1}}\cdot\frac{1}{|A|_{\max}^2}\notag\\
&+\sqrt{ \frac{5m-1}{4m}+ \frac{ m+\frac{1}{2}}{|A|_{\max}^2} + \frac{m}{4(m-1)} \frac{1}{|A|_{\max}^4}}\ \Bigg)\notag\\
\leq & \Bigg(\frac{1}{2}\sqrt{\frac{m-1}{m}}+\frac{1}{2}\sqrt{\frac{1}{m(m-1)}} \notag\\
&+\sqrt{\frac{5m-1}{4m} +1+\frac{1}{2m} +\frac{1}{4m(m-1)}  }   \ \Bigg)|A|_{\max}\notag\\
= &\frac{\sqrt{m}+\sqrt{9m-8}}{2\sqrt{m-1}} |A|_{\max}<2\sqrt{\frac{m}{m-1}}\cdot|A|_{\max},\notag
\end{align}
 We substitute it into \eqref{final estimate 2 phi} and note that $$2\sqrt{\frac{m}{m-1}}\cdot\sqrt{m} >1+\sqrt{m+1}.$$ This  completes the proof of the proposition.
\end{proof}

 Proposition \ref{function in tubular} also applies to $\Omega_2$. If $\phi=1$, then  we immediately have the following corollary:
 \begin{cor}\label{area curvature}
 \begin{align}
\textup{Volume} (\Sigma) \leq & 2\sqrt{\frac{m}{m-1}}\max\{|A|_{\max}, \sqrt{m} \}\min\{\textup{Volme} (\Omega_1), \textup{Volume}(\Omega_2) \}\notag\\
\leq& \sqrt{\frac{m}{m-1}}\max\{|A|_{\max}, \sqrt{m} \}\textup{Volume}(S^{m+1}). \notag
 \end{align}
 
 \end{cor}

Second, we also need the volume  growth of $\Sigma$, which can be found  in  Brendle \cite[Theorem~2.1]{Brendle2023minimal} or Colding and Minicozzi \cite[p.~24]{colding2011minimal}. Although their versions are minimal hypersurfaces in Euclidean space, with a slight modification, we can realize it  for minimal hypersurfaces in the unit  sphere.

 \begin{prop}\label{volume growth}
For any point $x_0$ in  $\Sigma$, if we denote by  $\rho $  the distance in $S^{m+1}$ to $x_0$ and let $$ B_s(x_0) =\{x\in S^{m+1} | \rho (x) \leq s  \} ,$$
then 
\begin{align}
 \int_{ B_s (x_0) \cap \Sigma} \cos \rho \ d\sigma  \leq\left( \int_{ B_{\frac{\pi}{2}} (x_0) \cap \Sigma} \cos \rho \ d\sigma \right)  \sin ^{m} s,        \  \forall\  0 \leq s < \frac{\pi}{2}.\notag
 \end{align}
\end{prop}
\begin{proof}
Since by  \eqref{definition},
\begin{align}
  \Delta ^{\Sigma}  \cos \rho=\Delta^{\Sigma} \langle x_0, F \rangle =-m \langle x_0, F \rangle=-m \cos \rho, \notag
  \end{align}

the divergence of the vector field $ -\frac{\nabla^{\Sigma} \cos \rho}{m \sin^m \rho}  $ is given by 
\begin{align}
 \textup{div} ^{\Sigma}\left(-\frac{\nabla^{\Sigma} \cos \rho}{m\sin^m \rho}\right)=
 \frac{-\Delta ^{\Sigma} \cos \rho}{ m\sin ^m \rho}+\frac{\langle \nabla^{\Sigma}\sin \rho, \nabla^{\Sigma} \cos \rho\rangle}{ \sin^{m+1} \rho}=\frac{\cos \rho \langle \nabla \rho , \mathbf{n} \rangle ^2  }{\sin ^m \rho }.
\end{align}
By Sard's theorem, there exists a dense subset $S$ of $[0,\frac{\pi}{2}]$ such that  $\partial B_{t} (x_0) $ meets  $\Sigma$ transversally for  every $t\in S $. We choose $t_1,t_2 \in S$ such that $t_1< t_2$. Applying the divergence  theorem to $-\frac{\nabla^{\Sigma} \cos \rho}{m \sin^m \rho}$ on $\left(B_{t_2} (x_0) \setminus B_{t_1} (x_0)\right) \cap \Sigma $ gives 
\begin{align}\label{first divergence}
 &\frac{1}{m \sin^m t_1} \int_{\partial B_{t_1}(x_0) \cap \Sigma} \langle\nabla^{\Sigma} \cos \rho, \frac{\nabla^{\Sigma} \rho}{|\nabla^{\Sigma} \rho|} \rangle -  \frac{1}{m \sin^m t_2} \int_{\partial B_{t_2}(x_0) \cap \Sigma} \langle \nabla^{\Sigma} \cos \rho, \frac{\nabla^{\Sigma} \rho}{|\nabla^{\Sigma} \rho|} \rangle\notag\\
 = &\int_{\left(B_{t_2} (x_0) \setminus B_{t_1} (x_0)\right) \cap \Sigma }  \textup{div}^{\Sigma} \left(-\frac{\nabla^{\Sigma} \cos \rho}{m\sin^m \rho}\right) = \int_{\left(B_{t_2} (x_0) \setminus B_{t_1} (x_0)\right) \cap \Sigma } \frac{\cos \rho \langle \nabla \rho , \mathbf{n} \rangle ^2  }{\sin ^m \rho } \geq 0. 
 \end{align}
Applying the divergence theorem to $\nabla^{\Sigma} \cos \rho $ on $B_{t_1} (x_0) \cap \Sigma$ and $B_{t_2} (x_0) \cap \Sigma$ gives
\begin{align}\label{second divergence t_1}
\frac{1}{m \sin^m t_1} \int_{\partial B_{t_1}(x_0) \cap \Sigma} \langle \nabla^{\Sigma} \cos \rho , \frac{\nabla^{\Sigma} \rho}{|\nabla^{\Sigma} \rho|}\rangle  &= \frac{1}{m \sin^m t_1} \int_{ B_{t_1}(x_0) \cap \Sigma}  \Delta^{\Sigma} \cos \rho \notag\\
&= - \frac{1}{ \sin^m t_1} \int_{ B_{t_1}(x_0) \cap \Sigma}   \cos \rho 
\end{align}
and 
\begin{align}\label{second divergence t_2}
\frac{1}{m \sin^m t_2} \int_{\partial B_{t_2}(x_0) \cap \Sigma} \langle \nabla^{\Sigma} \cos \rho , \frac{\nabla^{\Sigma} \rho}{|\nabla^{\Sigma} \rho|} \rangle &= \frac{1}{m \sin^m t_2} \int_{ B_{t_2}(x_0) \cap \Sigma}  \Delta^{\Sigma} \cos \rho \notag\\
&= - \frac{1}{ \sin^m t_2} \int_{ B_{t_2}(x_0) \cap \Sigma}   \cos \rho \ .
\end{align}
Combining \eqref{first divergence}, \eqref{second divergence t_1} and \eqref{second divergence t_2} yields 
\begin{align}
 \frac{1}{ \sin^m t_2} \int_{ B_{t_2}(x_0) \cap \Sigma}   \cos \rho \geq 
 \frac{1}{ \sin^m t_1} \int_{ B_{t_1}(x_0) \cap \Sigma}   \cos \rho. \end{align}
Then the proposition follows by letting $t_1 \rightarrow s ^{+} $  and $t_2 \rightarrow \left(\frac{\pi}{2}\right)^{-}$.

\end{proof}

\section{Proof of Theorem \ref{computation}}\label{computation'}
Let $u $  and $v$ be smooth functions   on  $\Omega_1$ and $\Omega_2$, respectively, and smooth up to the boundary. Let $f$ be a smooth function on $\Sigma$ such that $u|_{\Sigma}=v|_{\Sigma}=f$.
In this section, we always assume that $ \Omega_1$ and $\Omega_2 $  are the regions in $S^{m+1}$ and $\Sigma=\partial \Omega_1=\Omega_2$ is  minimal $(i.e., H=0)$. Our goal is to compute  $\Delta |D^2 u|^2, \Delta |D^2 v|^2$ and $\int_{\Omega_1}\Delta |D^2 u|^2, \int_{\Omega_2}\Delta |D^2 v|^2$. Here, in view of the particular setting of  $S^{m+1}$, we are going to use a  slightly simpler and clearer approach than general computation of derivative of  tensors. Its advantage is that we can directly use Reilly's  formula to compute the integrals $\int_{\Omega_1}\Delta |D^2 u|^2$ and $\int_{\Omega_2}\Delta |D^2 v|^2$ (see Lemma \ref{integral  computations}).  We will use Lemma \ref{some diffential computations}, Proposition \ref{Reilly} and Lemma \ref{integral  computations} to prove Theorem \ref{computation}. Furthermore, before proving Theorem \ref{computation}, for convenience, we still  only consider the case of $\Omega_1$ and the case of $\Omega_2$ is analogous.
We use $\frac{\partial}{\partial x_1},\cdots, \frac{\partial}{\partial x_{m+2}}$ to denote a standard  orthogonal  frame on $\mathbb{R}^{m+2}$.
Then 
$$ D^{\mathbb{R}^{m+2}}\frac{\partial}{\partial x_{\alpha} }=0,$$
$$\nabla u= \sum\limits_{\alpha=1}^{m+2}\langle \nabla u, \frac{\partial}{\partial x_\alpha}\rangle \frac{\partial}{\partial x_\alpha} $$
and 
$$ | \nabla u|^2= \sum\limits_{\alpha=1}^{m+2}\langle \nabla u, \frac{\partial}{\partial x_\alpha}\rangle^2,$$
where $D^{\mathbb{R}^{m+2}}$ is  the Levi-Civita connection of $\mathbb{R}^{m+2}.$

 We denote by $X$ the position vector in $\mathbb{R}^{m+2}$, and by $x_\alpha\  (\alpha=1,\cdots,m+2)$ its coordinate components. Then $X|_{\Sigma}=F$ and 
$$X=\sum\limits_{\alpha=1}^{m+2} x_{\alpha} 
\frac{\partial}{\partial x_{\alpha}}.$$
Furthermore, for any $X \in S^{m+1}$ and any smooth vector field $\eta$ on $S^{m+1}$,
\begin{align}
    \sum\limits_{\alpha=1}^{m+2} x_{\alpha} ^2=1, \langle X, \eta (X) \rangle=0.\notag
    \end{align}

Let $R$   denote  the curvature tensor of $S^{m+1}$. Then  for any smooth vector fields $ \eta_1,\eta_2,\eta_3, \eta_4$ on $S^{m+2}$,
\begin{align}\label{curvature}
 R(\eta_1, \eta_2, \eta_3,\eta_ 4) &= \langle -D_{\eta_1} D_{\eta_2} \eta_3  +D_{\eta_2} D_{\eta_1} \eta_3  +D_{[\eta_1,\eta_2 ]} \eta_3 ,\eta _4\rangle \notag\\
 &=\langle \eta_1 ,\eta_3 \rangle  \cdot \langle \eta_2, \eta_4 \rangle-\langle \eta_1,\eta_4 \rangle \cdot \langle \eta_2, \eta_3 \rangle.
\end{align}

For $ \langle \nabla u,\frac{\partial}{\partial x_{\alpha}} \rangle $, we have
\begin{lem}\label{some diffential computations}
In $\Omega_1$,
\begin{align}\label{gradient computation}
\sum\limits_{\alpha=1}^{m+2} |\nabla \langle \nabla u, \frac{\partial}{\partial x_{\alpha}} \rangle |^2  = |D^2u|^2+|\nabla u|^2,
\end{align}
\begin{align}\label{laplacian computation}
\Delta\langle \nabla u, \frac{\partial}{\partial x_{\alpha}} \rangle =\langle \nabla \Delta u, \frac{\partial}{\partial x_{\alpha}} \rangle-2x_{\alpha} \Delta u +(m-1) \langle \nabla u,\frac{\partial}{\partial x_{\alpha}} \rangle, \forall 1\leq \alpha \leq m+2,
\end{align}
\begin{align}\label{hession computation}
     \sum\limits_{\alpha=1}^{m+2} |D^2\langle \nabla u, \frac{\partial}{\partial x_{\alpha}} \rangle|^2=|D^3 u|^2 +4 |D^2 u|^2 -(m-1)|\nabla u|^2 -2\langle \nabla \Delta u,\nabla u  \rangle.
     \end{align}
\end{lem}
\begin{proof}
It  suffices to show that the conclusion holds for  $\Omega_1\setminus \Sigma$ since $u$ is smooth up the boundary  on $\Omega_1$.

Fix $p \in \Omega_1\setminus \Sigma$ and choose a local orthonormal frame $\{\tilde{e_i}\}_{i=1}^{m+1}$ near $p$
such that 
\begin{align}
 D_{\tilde{e_i}} \tilde{e_j} (p)=0, \ 1\leq i,j\leq m+1 \notag
\end{align}

Then at $p$, for each $i$, we have 
\begin{align}
\tilde{e_i}\langle \nabla u, \frac{\partial}{\partial x_{\alpha}}\rangle
&=\langle D^{\mathbb{R}^{m+2}}_{\tilde{e_i}} \nabla u ,\frac{\partial}{\partial x_{\alpha}}\rangle\notag\\
&=\langle D_{\tilde{e_i}} \nabla u,\frac{\partial}{\partial x_{\alpha}}\rangle + x_{\alpha} \langle D^{\mathbb{R}^{m+2}}_{\tilde{e_i}} \nabla u , X \rangle \notag \\
&=\langle D_{\tilde{e_i}} \nabla u,\frac{\partial}{\partial x_{\alpha}}\rangle - x_{\alpha} \langle D^{\mathbb{R}^{m+2}}_{\tilde{e_i}}  X,\nabla u\rangle \notag \\
&=\langle D_{\tilde{e_i}} \nabla u,\frac{\partial}{\partial x_{\alpha}}\rangle - x_{\alpha} \langle \nabla u, \tilde{e_i}\rangle, \label{first derivate}
\end{align}
which gives 
\begin{align}
&\sum\limits_{\alpha=1}^{m+2} |\nabla \langle \nabla u, \frac{\partial}{\partial x_{\alpha}} \rangle |^2 \notag\\ = &\sum\limits_{\alpha=1}^{m+2}\sum\limits_{i=1}^{m+1} \left(\langle D_{\tilde{e_i}} \nabla u,\frac{\partial}{\partial x_{\alpha}}\rangle - x_{\alpha} \langle \nabla u, \tilde{e_i}\rangle \right)^2 \notag\\
&= \sum\limits_{\alpha=1}^{m+2}\sum\limits_{i=1}^{m+1}\Big(\langle D_{\tilde{e_i}} \nabla u,\frac{\partial}{\partial x_{\alpha}}\rangle^2-2x_{\alpha} \langle D_{\tilde{e_i}} \nabla u,\frac{\partial}{\partial x_{\alpha}}\rangle \langle \nabla u, \tilde{e_i}\rangle  +x_{\alpha}^2 \langle \nabla u, \tilde{e_i} \rangle^2 \Big)\notag\\
&=|D^2u|^2+|\nabla u|^2-2\sum_{i=1}^{m+1} \langle D_{\tilde{e_i}} \nabla u, X \rangle  \langle \nabla u, \tilde{e_i} \rangle\notag\\
&=|D^2u|^2+|\nabla u|^2 \notag.
\end{align}
Fixing $i$ and $j$, at $p$, we have 
\begin{align}
  &\tilde{e_j}\tilde{e_i} \langle \nabla u, \frac{\partial}{\partial x_{\alpha}} \rangle\notag\\=&\tilde{e_j} \left(\langle D_{\tilde{e_i}} \nabla u,\frac{\partial}{\partial x_{\alpha}}\rangle - x_{\alpha} \langle \nabla u, \tilde{e_i}\rangle \right)\label{equ}\\
  =& \tilde{e_j}\langle D_{\tilde{e_i}} \nabla u,\nabla x_{\alpha} \rangle-
  \langle \nabla x_{\alpha},\tilde{e_j} \rangle \langle \nabla u,\tilde{e_i} \rangle -x_{\alpha} \langle D_{\tilde{e_j}} \nabla u,\tilde{e_i} \rangle\notag\\
  =& \tilde{e_j}\langle D_{\nabla x_{\alpha}} \nabla u,\tilde{e_i} \rangle-
  \langle \nabla x_{\alpha},\tilde{e_j} \rangle \langle \nabla u,\tilde{e_i} \rangle -x_{\alpha} \langle D_{\tilde{e_j}} \nabla u,\tilde{e_i} \rangle\notag\\
  =&\langle D_{\tilde{e_j}} D_{\nabla x_{\alpha}} \nabla u,\tilde{e_i} \rangle-\langle \nabla x_{\alpha},\tilde{e_j} \rangle \langle \nabla u,\tilde{e_i} \rangle -x_{\alpha} \langle D_{\tilde{e_j}} \nabla u,\tilde{e_i} \rangle\notag\\
  =& \langle D_{\nabla x_{\alpha}} D_{\tilde{e_j}} \nabla u,\tilde{e_i} \rangle-\langle D_{[\nabla x_{\alpha},\tilde{e_j}]} \nabla u,\tilde{e_i} \rangle +R(\nabla x_{\alpha},\tilde{e_j},\nabla u,\tilde{e_i} ) \notag\\
  &-\langle \nabla x_{\alpha},\tilde{e_j} \rangle \langle \nabla u,\tilde{e_i} \rangle  -x_{\alpha} \langle D_{\tilde{e_j}} \nabla u,\tilde{e_i} \rangle \notag\\
  =& \langle\nabla x_{\alpha}, \nabla \left( D^2 u( \tilde{e_i},\tilde{e_j})\right)\rangle+ \langle D_{D_{\tilde{e_j}} \nabla x_{\alpha}} \nabla u,\tilde{e_i} \rangle+ \delta_{ij} \langle \nabla x_{\alpha} , \nabla u\rangle \notag\\
  &-\langle \nabla x_{\alpha}, \tilde{e_i} \rangle  \langle \nabla u  ,\tilde{e_j}\rangle -\langle \nabla x_{\alpha},\tilde{e_j} \rangle \langle \nabla u,\tilde{e_i} \rangle  -x_{\alpha} \langle D_{\tilde{e_j}} \nabla u,\tilde{e_i} \rangle  \label{ineq} \\
 =& \langle\nabla x_{\alpha}, \nabla \left(D^2 u( \tilde{e_i},\tilde{e_j})\right)\rangle+\sum\limits_{l=1}^{m+1}\langle D_{\tilde{e_j}} \nabla x_{\alpha},\tilde{e_l} \rangle \cdot D^2 u(\tilde{e_i},\tilde{e_l}) + \delta_{ij} \langle \nabla x_{\alpha} , \nabla u\rangle \notag\\
  &-\langle \nabla x_{\alpha}, \tilde{e_i} \rangle  \langle \nabla u  ,\tilde{e_j}\rangle -\langle \nabla x_{\alpha},\tilde{e_j} \rangle \langle \nabla u,\tilde{e_i} \rangle  -x_{\alpha} D^2 u(\tilde{e_i},\Tilde{e_j} )\notag\\
   =& \langle \nabla \left(D^2 u( \tilde{e_i},\tilde{e_j}),  \nabla x_{\alpha}\right)\rangle+\sum\limits_{l=1}^{m+1}\langle D^{ \mathbb{R}^{m+2}}_{\tilde{e_j}} \left(\frac{\partial}{\partial x_{\alpha}}-x_{\alpha} X\right),\tilde{e_l} \rangle \cdot D^2 u(\tilde{e_i},\tilde{e_l})  \notag\\
  &+ \delta_{ij} \langle  \nabla u, \nabla x_{\alpha}\rangle-\langle \nabla x_{\alpha} , \tilde{e_i} \rangle  \langle \nabla u  ,\tilde{e_j}\rangle -\langle \nabla x_{\alpha},\tilde{e_j} \rangle \langle \nabla u,\tilde{e_i} \rangle  -x_{\alpha}D^2 u(\tilde{e_i},\Tilde{e_j} )\notag\\
  =& \langle \nabla \left(D^2 u( \tilde{e_i},\tilde{e_j})\right),  \frac{\partial}{\partial x_{\alpha}}\rangle-2x_{\alpha}D^2 u(\tilde{e_i},\Tilde{e_j})+ \delta_{ij} \langle  \nabla u, \nabla x_{\alpha} \rangle \notag\\
  &-\langle\frac{\partial}{\partial x_{\alpha}}, \tilde{e_i} \rangle  \langle \nabla u  ,\tilde{e_j}\rangle -\langle\frac{\partial}{\partial x_{\alpha}},\tilde{e_j} \rangle \langle \nabla u,\tilde{e_i} \rangle \notag,
  \end{align}
where in \eqref{equ}, we use \eqref{first derivate} and 
in \eqref{ineq}, we use the curvature of $S^{m+1}$ (see \eqref{curvature}):
$$R(\nabla x_{\alpha},\tilde{e_j},\nabla u,\tilde{e_i})=\delta_{ij}\langle \nabla x_{\alpha} ,\nabla u \rangle-\langle \nabla x_{\alpha},\tilde{e_i}\rangle \langle \nabla u,\tilde{e_j} \rangle. $$

This gives 
\begin{align}
\Delta\langle \nabla u, \frac{\partial}{\partial x_{\alpha}} \rangle =&\sum\limits_{i=1}^{m+1}\tilde{e_i}\tilde{e_i} \langle \nabla u, \frac{\partial}{\partial x_{\alpha}} \rangle \notag\\
=& \langle \nabla \Delta u, \frac{\partial}{\partial x_{\alpha}} \rangle-2x_{\alpha} \Delta u +(m+1) \langle \nabla u,\frac{\partial }{\partial x_{\alpha}}\rangle -2\sum\limits_{i=1}^{m+1} \langle \nabla u, 
\nabla x_{\alpha}\rangle \notag\\
=&\langle \nabla \Delta u, \frac{\partial}{\partial x_{\alpha}} \rangle-2x_{\alpha} \Delta u +(m-1) \langle \nabla u,\frac{\partial}{\partial x_{\alpha}} \rangle \notag
\end{align}
and 
\begin{align}
 &\sum\limits_{\alpha=1}^{m+2} |D^2\langle \nabla u, \frac{\partial}{\partial x_{\alpha}} \rangle|^2 \notag\\
 =&\sum\limits_{\alpha=1}^{m+2}\sum\limits_{i,j=1}^{m+1} \left(\tilde{e_j}\tilde{e_i} \langle \nabla u, \frac{\partial}{\partial x_{\alpha}} \rangle\right)^2\notag\\
 =& |D^3 u|^2+ 4|D^2u |+(m+1) |\nabla u|^2 +2(m+1)|\nabla u |^2 \notag\\
 &-4\sum\limits_{i,j=1}^{m+1} D^2 u (\tilde{e_i}, \tilde{e_j} )\cdot \langle \nabla \left(D^2 u (\tilde{e_i},\tilde{e_i})\right), X\rangle +2\langle\nabla \Delta u,\nabla u\rangle \notag\\
 &-2\sum\limits_{i,j=1}^{m+1}\langle \nabla\left( D^2u(\tilde{e_i} ,\tilde{e_j})\right),\tilde{e_i}\rangle \cdot\langle \nabla u,\tilde{e_j} \rangle \notag\\
&-2\sum\limits_{i,j=1}^{m+1} \langle \nabla\left( D^2u(\tilde{e_i} ,\tilde{e_j})\right), \tilde{e_j} \rangle\cdot \langle \nabla u,\tilde{e_i} \rangle \notag\\
 &-4 \Delta u \cdot \langle\nabla u,X \rangle+ 4 \sum_{i,j=1}^{m+1} D^2 u(\tilde{e_i},\tilde{e_j} ) \cdot \langle X,\tilde{e_i} \rangle \langle \nabla u,\tilde{e_j} \rangle \notag\\
&+4\sum_{i,j=1}^{m+1} D^2 u(\tilde{e_i},\tilde{e_j} )\cdot \langle X,\tilde{e_j}\rangle \langle \nabla u,\tilde{e_i}  \rangle-4|\nabla u|^2+2 |\nabla u|^2, \notag\\
 =&|D^3 u|^2 +4 |D^2 u|^2 +(3m+1)|\nabla u|^2 +2\langle \nabla \Delta u,\nabla u \rangle \notag\\
&-4 \sum\limits_{i,j=1}^{m+1}\langle D_{\tilde{e_i}} D_{\tilde{e_j}}\nabla u,\tilde{e_i} \rangle\cdot \langle \nabla u,\tilde{e_j} \rangle \notag
\end{align}
we continue from the previous page:
\begin{align}
=& |D^3 u|^2 +4 |D^2 u|^2 +(3m+1)|\nabla u|^2 +2\langle \nabla \Delta u,\nabla u \rangle \notag\\
&-4 \sum\limits_{i,j=1}^{m+1} \left( \langle D_{\tilde{e_j}} D_{\tilde{e_i}}  \nabla u,\tilde{e_i}\rangle + R( \tilde{e_j},\tilde{e_i},\nabla u,\tilde{e_i} )\right) \cdot \langle\nabla u,\tilde{e_j}  \rangle\notag\\
=&|D^3 u|^2 +4 |D^2 u|^2 +(3m+1)|\nabla u|^2 +2\langle \nabla \Delta u,\nabla u \rangle \notag\\
&-4\sum\limits_{i=1}^{m+1}\langle D_{\nabla u} D_{\tilde{e_i}} \nabla u,\tilde{e_i} \rangle -4(m+1) |\nabla u |^2 +4|\nabla u |^2 \label{ineq'}\\
=& |D^3 u|^2 +4 |D^2 u|^2 -(m-1)|\nabla u|^2 -2\langle \nabla \Delta u,\nabla u  \rangle \notag,
\end{align}
where in \eqref{ineq'}, we use the curvature property of $S^{m+1}$ (see \eqref{curvature}):
$$R( \nabla u,\tilde{e_i},\nabla u,\tilde{e_i} )= |\nabla u|^2-\langle \nabla u,\tilde{e_i}\rangle^2.$$

\end{proof}

We also need the following well-known Reilly's formula:
\begin{prop}[\cite{Reilly1977apllication}] \label{Reilly} 
Under the assumptions of Theorem \ref{computation},
\begin{align}
&-2\int_{\Sigma} \langle \nabla^{\Sigma} \langle \nabla u,\mathbf{n} \rangle ,\nabla^{\Sigma} f \rangle \ d\sigma -\int_{\Sigma} A(\nabla^{\Sigma} f, \nabla^{\Sigma} f)\ d\sigma \notag\\
&=\int_{\Omega_1} |D^2 u|^2 +m\int_{\Omega_1} |\nabla u|^2-\int_{\Omega_1} (\Delta u)^2.\notag
\end{align}
\end{prop}
\begin{proof}
 The equalities  \eqref{laplacian computation}  and \eqref{gradient computation} give 
 \begin{align}
  &\frac{1}{2} \Delta |\nabla u|^2 \notag\\=&\frac{1}{2} \sum\limits_{\alpha=1}^{m+2} \Delta \langle \nabla u,\frac{\partial}{\partial x_{\alpha}} \rangle^2\notag\\
  =& \sum\limits_{\alpha=1}^{m+2}|\nabla \langle \nabla u,\frac{\partial}{\partial x_{\alpha}} \rangle|^2 +\sum\limits_{\alpha=1}^{m+2} \langle \nabla u,\frac{\partial}{\partial x_{\alpha}} \rangle \Delta \langle \nabla u,\frac{\partial}{\partial x_{\alpha}} \rangle \notag\\
  =& |D^2 u|^2 +|\nabla u|^2 +\sum\limits_{\alpha=1}^{m+2} \langle \nabla u, \frac{\partial}{\partial x_{\alpha}} \rangle \cdot\Big(\langle \nabla \Delta u, \frac{\partial}{\partial x_{\alpha}} \rangle-2x_{\alpha} \Delta u +(m-1) \langle \nabla u,\frac{\partial}{\partial x_{\alpha}} \rangle\Big) \notag\\
  =& |D^2 u|^2+m|\nabla u|^2 +\langle \nabla \Delta u , \nabla u\rangle -\Delta u \cdot  \langle \nabla u ,X \rangle \notag\\
  =&|D^2 u|^2+m|\nabla u|^2 +\langle \nabla \Delta u , \nabla u\rangle.  \end{align}
This is  Bochner's  formula.

Applying the divergence theorem yields 
\begin{align}\label{first div}
-\int_{\Sigma} \langle D_{\mathbf{n}} \nabla u ,\nabla u   \rangle=&\int_{\Omega_1} |D^2 u|^2+m \int_{\Omega_1}|\nabla u|^2 + \int_{\Omega_1} \langle \nabla \Delta u , \nabla u\rangle \notag\\
=&  \int_{\Omega_1} |D^2 u|^2+m \int_{\Omega_1}|\nabla u|^2 - \int_{\Sigma}
\langle \nabla u, \mathbf{n} \rangle \Delta u -\int_{\Omega_1} (\Delta u)^2,
\end{align}
where we note that $\mathbf{n}$ is the  inward-pointing  unit normal vector field on $\Sigma$.

For the term 
$\langle D_{\mathbf{n}} \nabla u ,\nabla u   \rangle$, we have
\begin{align}\label{local com}
&\langle D_{\mathbf{n}} \nabla u ,\nabla u   \rangle= \langle D_{\nabla u}  \nabla u,\mathbf{n} \rangle \notag\\
=& \langle \nabla u   , \mathbf{n} \rangle \cdot \langle D_{\mathbf{n}}\nabla u,\mathbf{n} \rangle +\langle D_{\nabla^{\Sigma} f} \nabla u,\mathbf{n}  \rangle \notag\\
=& \langle \nabla u   , \mathbf{n} \rangle \cdot\left(\Delta u-\sum_{i=1}^{m} \langle D_{e_i} \nabla u,e_i\rangle \right) +\langle \nabla^{\Sigma} f, \nabla^{\Sigma} \langle \nabla u,\mathbf{n}\rangle    \rangle+ A(\nabla^{\Sigma} f, \nabla^{\Sigma} f) \notag\\
=&\langle \nabla u   , \mathbf{n} \rangle \left(\Delta u+m H \cdot\langle \nabla u,\mathbf{n} \rangle -\Delta^{\Sigma}  f \right)+ \langle \nabla^{\Sigma} f, \nabla^{\Sigma} \langle \nabla u,\mathbf{n}\rangle    \rangle+ A(\nabla^{\Sigma} f, \nabla^{\Sigma} f) \notag\\
=& \langle \nabla u   , \mathbf{n} \rangle\Delta u-  \langle \nabla u   , \mathbf{n} \rangle \Delta^{\Sigma}  f + \langle \nabla^{\Sigma} f, \nabla^{\Sigma} \langle \nabla u,\mathbf{n}\rangle    \rangle+ A(\nabla^{\Sigma} f, \nabla^{\Sigma} f),
\end{align}
where $\{e_i\}$ is a local orthonormal basis on $\Sigma$, $u|_{\Sigma}=f$ and  the mean curvature $H$ vanishes identically.

Combining \eqref{first div} and \eqref{local com}  and noting that
\begin{align}
 \int_{\Sigma}  \langle \nabla u,\mathbf{n}\rangle \Delta^{\Sigma} f=-\int_{\Sigma} \langle \nabla^{\Sigma} \langle \nabla u,\mathbf{n} \rangle ,\nabla^{\Sigma} f \rangle \notag
 \end{align}
by integration by parts, we  can get the conclusion of this proposition. 
\end{proof}
Since the above Reilly's formula holds for any smooth function, this formula applies to $\langle\nabla u,\frac{\partial}{\partial x_{\alpha}}    \rangle$.
\begin{lem}\label{integral  computations}
Under the assumptions of Theorem \ref{computation},
\begin{align}
&-2\int_{\Sigma} A\left(\nabla^{\Sigma}(\Delta u-2\Delta^{\Sigma} f ),\nabla^{\Sigma} f\right)+2 \int_{\Sigma} \langle \nabla u,\mathbf{n} \rangle \cdot \Delta^{\Sigma} (\Delta u-2\Delta^{\Sigma}f) \notag\\
&-4m\int_{\Sigma} \langle \nabla u,\mathbf{n} \rangle \cdot \Delta^{\Sigma} f+(3m-1)\int_{\Sigma} A(\nabla^{\Sigma} f,\nabla^{\Sigma} f )\notag\\
&-5\int_{\Sigma}A\Big(\big( D_{\mathbf{n}}\nabla u\big)^{\top},\big( D_{\mathbf{n}}\nabla u\big)^{\top}\Big)+2\int_{\Sigma} (\Delta u-\Delta^{\Sigma} f )\cdot \langle A,F^{\star}\big(D^2u \big)\rangle \notag\\
&-\int_{\Sigma} \textup{Trace}_{\Sigma} \Bigg (A\Big(\big( D_{()} \nabla u \big)^{\top} ,\big( D_{()} \nabla u \big)^{\top}\Big)\Bigg)\notag\\
=&\int_{\Omega_1}|D^3 u|^2 -2m(m+1)\int_{\Omega_1}|\nabla u|^2 \notag\\
&-\int_{\Omega_1}|\nabla \Delta u |^2+m\int_{\Omega_1}(\Delta u)^2 -2m\int_{\Omega_1}\langle\nabla \Delta u,\nabla u    \rangle.\notag
\end{align}
\end{lem}
\begin{proof}
First, for each $\langle \nabla u, \frac{\partial}{\partial x_{\alpha}} \rangle$, applying Proposition \ref{Reilly} gives 
\begin{align}
 & -2 \int_{\Sigma} \left\langle \nabla^{\Sigma} \Big\langle \nabla\langle \nabla u,\frac{\partial}{\partial x_{\alpha}} \rangle, \mathbf{n}  \Big \rangle , \nabla^{\Sigma} \langle \nabla u,\frac{\partial}{\partial x_{\alpha}}  \rangle   \right\rangle -\int_{\Sigma} A\left(\nabla^{\Sigma}\langle \nabla u,\frac{\partial}{\partial x_{\alpha}}  \rangle,\nabla^{\Sigma}\langle \nabla u,\frac{\partial}{\partial x_{\alpha}}  \rangle\right) \notag\\
&= \int_{\Omega_1} |D^2 \langle \nabla u,\frac{\partial}{\partial x_{\alpha}} \rangle|^2 +m \int_{\Omega_1} |\nabla \langle \nabla u,\frac{\partial}{\partial x_{\alpha}} \rangle|^2-\int_{\Omega_1} \left(\Delta\langle \nabla u,\frac{\partial}{\partial x_{\alpha}} \rangle\right)^2. \label{Reilly application}
\end{align}
Using \eqref{laplacian computation}, we have
\begin{align}
\sum_{\alpha=1} ^{m+2}\left(\Delta\langle \nabla u,\frac{\partial}{\partial x_{\alpha}} \rangle\right)^2 =& \sum_{\alpha=1}^{m+2}\left(\langle \nabla \Delta u, \frac{\partial}{\partial x_{\alpha}} \rangle-2x_{\alpha} \Delta u +(m-1) \langle \nabla u,\frac{\partial}{\partial x_{\alpha}} \rangle\right)^2\notag\\
=& |\nabla \Delta u|^2 +4 (\Delta u )^2 +(m-1)^2 | \nabla u|^2-4 \Delta u \cdot \langle  \nabla \Delta u ,X \rangle \notag\\
&+2 (m-1) \langle \nabla \Delta u, \nabla u \rangle-4(m-1) \Delta u \cdot \langle \nabla u,X\rangle\notag\\
=&|\nabla \Delta u|^2 +4 (\Delta u )^2 +(m-1)^2 | \nabla u|^2+ 2 (m-1) \langle \nabla \Delta u, \nabla u \rangle .\label{squared laplace}
\end{align}
 Fix $q\in \Sigma$ and choose a local  orthonormal frame $\{e_i\}_{i=1}^{m}$ near $q$ on $\Sigma$ such that 
$$ D^{\Sigma}_{e_i} e_j (q)=0, \ [e_i,e_j](q)=0,\ A(e_i,e_j)(q) =\mu_i \delta_{ij}, \  1\leq i,j \leq m.$$
At $q$, we compute
 \begin{align}
&\sum_{\alpha=1}^{m+2}\left\langle \nabla^{\Sigma} \Big\langle \nabla\langle \nabla u,\frac{\partial}{\partial x_{\alpha}} \rangle, \mathbf{n}  \Big \rangle , \nabla^{\Sigma} \langle \nabla u,\frac{\partial}{\partial x_{\alpha}}  \rangle   \right\rangle \notag\\
=& \sum_{\alpha=1}^{m+2} \sum_{i=1}^{m} e_i \Big\langle \nabla\langle \nabla u,\frac{\partial}{\partial x_{\alpha}} \rangle, \mathbf{n}  \Big \rangle \cdot e_i \langle \nabla u,\frac{\partial}{\partial x_{\alpha}}  \rangle\notag\\
=&\sum_{\alpha=1}^{m+2} \sum_{i=1}^{m} e_i \left( \langle D^{\mathbb{R}^{m+2}}_{\mathbf{n}} \nabla u ,  \frac{\partial}{\partial x_{\alpha}}\rangle    \right) \cdot \langle D^{\mathbb{R}^{m+2}}_{e_i} \nabla u , \frac{\partial}{\partial x_{\alpha}}\rangle \notag\\
 =& \sum_{\alpha=1}^{m+2} \sum_{i=1}^{m} e_i\left( 
\langle D_{\mathbf{n}} \nabla u ,\frac{\partial}{\partial x_{\alpha}} \rangle -x_{\alpha}  \langle \nabla u,\mathbf{n}\rangle  \right) \cdot \left( \langle  D_{e_i}\nabla u,\frac{\partial}{\partial x_{\alpha}}  \rangle -x_{\alpha} \langle \nabla u,e_i \rangle  \right) \label{equ'}\\
=& \sum_{\alpha=1}^{m+2} \sum_{i=1}^{m} \left(\langle D^{\mathbb{R}^{m+2}}_{e_i} D_{\mathbf{n}}\nabla u ,  \frac{\partial}{\partial x_{\alpha}}\rangle -\langle  e_i, \frac{\partial}{\partial x_{\alpha}} \rangle \langle \nabla u, \mathbf{n}\rangle- x_{\alpha}\langle D_{e_i} \nabla u, \mathbf{n}\rangle -x_{\alpha} \langle \nabla u,D_{e_i} \mathbf{n}\rangle\right)\notag\\
&\qquad \quad\times \left( \langle  D_{e_i}\nabla u,\frac{\partial}{\partial x_{\alpha}}  \rangle -x_{\alpha} \langle \nabla u,e_i \rangle  \right)\notag\\
=&\sum_{\alpha=1}^{m+2} \sum_{i=1}^{m}\left(\langle D_{e_i} D_{\mathbf{n}} \nabla u, \frac{\partial}{\partial x_{\alpha}} \rangle-\langle  e_i, \frac{\partial}{\partial x_{\alpha}} \rangle \langle \nabla u, \mathbf{n}\rangle- 2x_{\alpha}\langle D_{e_i} \nabla u, \mathbf{n}\rangle +\mu_i x_{\alpha} \langle \nabla u,e_i\rangle\right)\notag\\
&\quad\qquad\times \left( \langle  D_{e_i}\nabla u,\frac{\partial}{\partial x_{\alpha}}  \rangle -x_{\alpha} \langle \nabla u,e_i \rangle  \right) ,\label{equ''}
\end{align}
where in \eqref{equ'}, we use 
\begin{align}
\langle D^{\mathbb{R}^{m+2}}_{\mathbf{n}} \nabla u ,  \frac{\partial}{\partial x_{\alpha}}\rangle=& \langle D_{\mathbf{n}} \nabla u ,\frac{\partial}{\partial x_{\alpha}} \rangle +x_{\alpha} \langle D_{\mathbf{n}} \nabla u ,X \rangle \notag\\
=&\langle D_{\mathbf{n}} \nabla u ,\frac{\partial}{\partial x_{\alpha}} \rangle-x_{\alpha}\langle D_{\mathbf{n}} X,\nabla u  \rangle \notag\\
=&\langle D_{\mathbf{n}} \nabla u ,\frac{\partial}{\partial x_{\alpha}} \rangle -x_{\alpha}  \langle \nabla u,\mathbf{n}\rangle \notag;
\end{align}
in \eqref{equ''}, using an analogous
argument to the above  yields
\begin{align}
\langle D^{\mathbb{R}^{m+2}}_{e_i} D_{\mathbf{n}}\nabla u ,  \frac{\partial}{\partial x_{\alpha}}\rangle=\langle D_{e_i} D_{\mathbf{n}} \nabla u, \frac{\partial}{\partial x_{\alpha}} \rangle-x_{\alpha}\langle D_{e_i} \nabla u, \mathbf{n}\rangle
\end{align}
and use $ D_{e_i} \mathbf{n}=-\mu_i e_i$.

We will continue the computation from the previous page:
\begin{align}
&\sum_{\alpha=1}^{m+2}\left\langle \nabla^{\Sigma} \Big\langle \nabla\langle \nabla u,\frac{\partial}{\partial x_{\alpha}} \rangle, \mathbf{n}  \Big \rangle , \nabla^{\Sigma} \langle \nabla u,\frac{\partial}{\partial x_{\alpha}}  \rangle   \right\rangle \notag\\
=& \sum_{i=1}^{m} \Big(\langle D_{e_i} D_{\mathbf{n}} \nabla u,D_{e_i } \nabla u \rangle -\langle D_{e_i}  \nabla u,e_i\rangle \langle \nabla u, \mathbf{n}\rangle \notag\\
&\qquad-2 \langle D_{e_i} \nabla u ,X \rangle \langle D_{e_i} \nabla u,\mathbf{n} \rangle +\mu_i \langle D_{e_i} \nabla u,X \rangle  \langle\nabla u,e_i \rangle \notag\\
&\qquad - \langle D_{e_i} D_{\mathbf{n}} \nabla u, X\rangle \langle \nabla u,e_i\rangle + \langle e_i, X\rangle \langle \nabla u,e_i\rangle \langle \nabla u,\mathbf{n} \rangle \notag\\
&\qquad+2\langle\nabla u,e_i \rangle\langle D_{e_i} \nabla u,\mathbf{n} \rangle  -\mu_i \langle \nabla u,e_i \rangle^2\Big)\notag\\
=& \sum_{i=1}^{m} \Big(\langle D_{e_i} D_{\mathbf{n}} \nabla u,D_{e_i } \nabla u \rangle -\mu_i \langle\nabla^{\Sigma}  f,e_i \rangle ^2  \Big) -
\langle\nabla u,\mathbf{n} \rangle \cdot \Delta^{\Sigma} f  \notag\\
&+mH \cdot \langle \nabla u,\mathbf{n}\rangle^2+2 \langle D_{\nabla^{\Sigma} f} \nabla u,\mathbf{n} \rangle \label{equ'''}\\
=& \sum_{i,j=1}^{m} \langle D_{e_i} \big(\langle D_{\mathbf{n}} \nabla u, e_j\rangle e_j\big), D_{e_i} \nabla u\rangle + \sum_{i=1}^{m} \langle D_{e_i}\big(\langle D_{\mathbf{n}}\nabla u,\mathbf{n} \rangle \mathbf{n}\big),D_{e_i} \nabla u \rangle \notag\\
& + A(\nabla^{\Sigma} f,\nabla^{\Sigma} f) +2\langle \nabla^{\Sigma} f, \nabla^{\Sigma}\langle\nabla u,\mathbf{n} \rangle \rangle-\langle \nabla u,\mathbf{n} \rangle\cdot \Delta^{\Sigma} f \label{equ''''}\\
=&\sum_{i,j=1}^{m} \Big(e_i \langle D_{e_j} \nabla u,\mathbf{n} \rangle \cdot \langle D_{e_j} \nabla u,e_i \rangle+ \langle D_{e_j} \nabla u,\mathbf{n}\rangle \langle D_{e_i} e_j, D_{e_i} \nabla u \rangle \Big)\notag\\
&+ \sum_{i=1}^{m} \Big(e_i(\Delta u-\sum_{j=1}^{m}\langle D_{e_j}\nabla u,e_j\rangle) \cdot \langle D_{e_i}\nabla u,\mathbf{n} \rangle \notag\\  
&+ (\Delta u-\sum_{j=1}^{m}\langle D_{e_j}\nabla u,e_j\rangle)\cdot\langle D_{e_i} \mathbf{n},D_{e_i} \nabla u\rangle   \Big) \notag\\
&+ A(\nabla^{\Sigma} f,\nabla^{\Sigma} f)+ 2\langle \nabla^{\Sigma} f, \nabla^{\Sigma}\langle\nabla u,\mathbf{n} \rangle \rangle-\langle \nabla u,\mathbf{n} \rangle\cdot \Delta^{\Sigma} f\notag
\end{align}
where in \eqref{equ'''}, we use $u|_{\Sigma}=f$ and 
\begin{align}
 \sum_{i=1}^{m} \langle D_{e_i} \nabla u, e_i\rangle=\Delta^{\Sigma} f-mH\cdot \langle \nabla u,\mathbf{n} \rangle \label{basic equ};
\end{align}
in \eqref{equ''''}, we use $H=0$ and 
\begin{align}
 \langle D_{\nabla^{\Sigma} f}\nabla u ,\mathbf{n}\rangle =   A(\nabla^{\Sigma} f,\nabla^{\Sigma} f)+ \langle \nabla^{\Sigma} f,\nabla^{\Sigma}\langle\nabla u,\mathbf{n} \rangle \rangle.\label{basic equ'}
 \end{align}
\begin{align}
  \sum_{i=1}^{m} \mu_i \langle\nabla^{\Sigma} f,e_i\rangle^2= A(\nabla^{\Sigma}f, \nabla^{\Sigma}f).\notag
  \end{align}
We  continue the computation from the previous page:
\begin{align}
&\sum_{\alpha=1}^{m+2}\left\langle \nabla^{\Sigma} \Big\langle \nabla\langle \nabla u,\frac{\partial}{\partial x_{\alpha}} \rangle, \mathbf{n}  \Big \rangle , \nabla^{\Sigma} \langle \nabla u,\frac{\partial}{\partial x_{\alpha}}  \rangle   \right\rangle \notag\\
=& \textup{div}^{\Sigma}(Y)-\sum_{i,j=1}^{m}\langle D_{e_j} \nabla u,\mathbf{n}\rangle\langle D_{e_i}D_{e_j} \nabla u,e_i \rangle\notag\\
&+\sum_{i,j=1}^{m}\Big( -\langle D_{e_j}\nabla u,\mathbf{n} \rangle \langle D_{e_j}\nabla u,D_{e_i}e_i \rangle +\langle D_{e_j} \nabla u,\mathbf{n}\rangle \langle D_{e_i} e_j, D_{e_i} \nabla u \rangle \Big) \notag\\
&+\sum_{i=1}^{m} \Big(e_i(\Delta u-\Delta^{\Sigma} f+ mH\langle\nabla u,\mathbf{n} \rangle    ) \cdot \langle D_{e_i}\nabla u,\mathbf{n} \rangle\notag\\
&\qquad \quad + (\Delta u-\Delta^{\Sigma} f+ mH\langle\nabla u,\mathbf{n} \rangle )\cdot\langle D_{e_i} \mathbf{n},D_{e_i} \nabla u\rangle   \Big)\label{equ'''''}\\
&+ A(\nabla^{\Sigma} f,\nabla^{\Sigma} f)+ 2\langle \nabla^{\Sigma} f, \nabla^{\Sigma}\langle\nabla u,\mathbf{n} \rangle \rangle-\langle \nabla u,\mathbf{n} \rangle\cdot \Delta^{\Sigma} f\notag\\
=& \textup{div}^{\Sigma}(Y)-\sum_{i,j=1}^{m}\langle D_{e_j} \nabla u,\mathbf{n}\rangle\langle D_{e_i}D_{e_j} \nabla u,e_i \rangle \notag  \\
& +\sum_{i=1}^{m} (\mu_i-mH)  \langle D_{e_i} \nabla u,\mathbf{n} \rangle^2
+ A(\nabla^{\Sigma} f,\nabla^{\Sigma} f)+ 2\langle \nabla^{\Sigma} f, \nabla^{\Sigma}\langle\nabla u,\mathbf{n} \rangle \rangle \notag\\
&-\langle \nabla u,\mathbf{n} \rangle\cdot \Delta^{\Sigma} f+\sum_{i=1}^{m} \Big(e_i(\Delta u-\Delta^{\Sigma} f+ mH\langle\nabla u,\mathbf{n} \rangle    ) \cdot \langle D_{e_i}\nabla u,\mathbf{n} \rangle\notag\\
&\qquad \qquad \qquad \qquad \qquad \  -\mu_i (\Delta u-\Delta^{\Sigma} f+ m H\langle\nabla u,\mathbf{n} \rangle )\cdot\langle D_{e_i} \nabla u,e_i \rangle \Big), \notag
\end{align}
where in \eqref{equ'''''}, $Y$  is a global smooth vector field on $\Sigma$, defined  by
\begin{align}\label{the definition of Y}
    Y=\sum_{i=1}^{m} \Big(\sum_{j=1}^{m} \langle D_{\overline{e_j}}\nabla u,\mathbf{n} \rangle \cdot \langle D_{\overline{e_j}} \nabla u,\overline{e_i}\rangle  \Big) \overline{e_i}
    \end{align}
 in any local orthonormal frame $\{\overline{e_l}\}_{l=1}^{m}$ on $\Sigma$, and we also use \eqref{basic equ}. 
 
We  continue the computation from the previous page:
\begin{align}
&\sum_{\alpha=1}^{m+2}\left\langle \nabla^{\Sigma} \Big\langle \nabla\langle \nabla u,\frac{\partial}{\partial x_{\alpha}} \rangle, \mathbf{n}  \Big \rangle , \nabla^{\Sigma} \langle \nabla u,\frac{\partial}{\partial x_{\alpha}}  \rangle   \right\rangle \notag\\
=& \textup{div}^{\Sigma}(Y)-\sum_{i,j=1}^{m}\langle D_{e_j} \nabla u,\mathbf{n}  \rangle \cdot \Big(\langle D_{e_j} D_{e_i}\nabla u,e_i\rangle +R(e_j,e_i,\nabla u,e_i)\Big)\notag\\
&+ \sum_{i=1}^{m}\mu_i \left (D^2 u(e_i,\mathbf{n})\right)^2+ A(\nabla^{\Sigma} f,\nabla^{\Sigma} f)+ 2\langle \nabla^{\Sigma} f, \nabla^{\Sigma}\langle\nabla u,\mathbf{n} \rangle \rangle\notag\\
&-\langle \nabla u,\mathbf{n} \rangle\cdot \Delta^{\Sigma} f+A\left(\nabla^{\Sigma}(\Delta u-\Delta^{\Sigma} f ),\nabla^{\Sigma} f\right)\notag\\
&+\langle\nabla^{\Sigma}(\Delta u-\Delta^{\Sigma} f),\nabla^{\Sigma} \langle \nabla u,\mathbf{n} \rangle\rangle    -\sum_{i=1}^{m} \mu_i(\Delta u-\Delta^{\Sigma} f )\cdot D^2 u(e_i,e_i)\label{last equ}\\
=& \textup{div}^{\Sigma}(Y)- \sum_{j=1}^{m}\langle D_{e_j} \nabla u,\mathbf{n}  \rangle \cdot e_j ( \Delta^{\Sigma} f-m H\langle\nabla u,\mathbf{n} \rangle)\notag\\
&+\sum_{i,j=1}^{m}\langle D_{e_j} \nabla u,\mathbf{n}  \rangle \cdot \langle D_{e_i} \nabla u, D_{e_j}e_i \rangle -\sum_{i,j=1}^{m}\langle D_{e_j} \nabla u,\mathbf{n}  \rangle \cdot \big(\langle\nabla u,e_j \rangle  -\delta_{ij}\langle \nabla u,e_i  \rangle  \big)\label{last equ'}\\
&+\sum_{i=1}^{m}\mu_i \left (D^2 u(e_i,\mathbf{n})\right)^2 +  A(\nabla^{\Sigma} f,\nabla^{\Sigma} f) + 2\langle \nabla^{\Sigma} f, \nabla^{\Sigma}\langle\nabla u,\mathbf{n} \rangle \rangle\notag\\
&-\langle \nabla u,\mathbf{n} \rangle\cdot \Delta^{\Sigma} f
+A\left(\nabla^{\Sigma}(\Delta u-\Delta^{\Sigma} f ),\nabla^{\Sigma} f\right)+\langle\nabla^{\Sigma}(\Delta u-\Delta^{\Sigma} f),\nabla^{\Sigma} \langle \nabla u,\mathbf{n} \rangle\rangle \notag\\
&-\sum_{i=1}^{m} \mu_i(\Delta u-\Delta^{\Sigma} f )\cdot D^2 u(e_i,e_i)\notag\\
=& \textup{div}^{\Sigma}(Y)+A\left(\nabla^{\Sigma}(\Delta u-2\Delta^{\Sigma} f ),\nabla^{\Sigma} f\right) +\langle\nabla^{\Sigma}(\Delta u-2\Delta^{\Sigma} f),\nabla^{\Sigma} \langle \nabla u,\mathbf{n} \rangle\rangle\notag\\
&+ (2-m) A(\nabla^{\Sigma} f,\nabla^{\Sigma} f)+ (3-m)\langle \nabla^{\Sigma} f, \nabla^{\Sigma}\langle\nabla u,\mathbf{n} \rangle \rangle-\langle \nabla u,\mathbf{n} \rangle\cdot \Delta^{\Sigma} f\label{last equ''}\\
&+2 \sum_{i=1}^{m}\mu_i \left (D^2 u(e_i,\mathbf{n})\right)^2- \sum_{i=1}^{m} \mu_i(\Delta u-\Delta^{\Sigma} f )\cdot D^2 u(e_i,e_i) \notag\\
=& \textup{div}^{\Sigma}(Y)+A\left(\nabla^{\Sigma}(\Delta u-2\Delta^{\Sigma} f ),\nabla^{\Sigma} f\right) +\langle\nabla^{\Sigma}(\Delta u-2\Delta^{\Sigma} f),\nabla^{\Sigma} \langle \nabla u,\mathbf{n} \rangle\rangle\notag\\
&+ (2-m) A(\nabla^{\Sigma} f,\nabla^{\Sigma} f)+ (3-m)\langle \nabla^{\Sigma} f, \nabla^{\Sigma}\langle\nabla u,\mathbf{n} \rangle \rangle-\langle \nabla u,\mathbf{n} \rangle\cdot \Delta^{\Sigma} f 
\notag\\
&+2 A\Big(\big( D_{\mathbf{n}}\nabla u\big)^{\top},\big(D_{\mathbf{n}}\nabla u\big)^{\top}\Big)-(\Delta u-\Delta^{\Sigma} f )\cdot                    \langle A,F^{\star}\big(D^2u \big)\rangle,\notag
\end{align}
where in \eqref{last equ}, we use $H=0$ and \eqref{basic equ'} ( substitute $f$ with  $ \Delta u-\Delta^{\Sigma} f$ ); in \eqref{last equ'}, we use \eqref{basic equ} and the curvature of $S^{m+1}$(see \eqref{curvature}):
$$R(e_j,e_i,\nabla u,e_i) =\langle \nabla u,e_j\rangle -\delta_{ij}\langle \nabla u,e_i\rangle;$$ 
in \eqref{last equ''}, one term  directly  uses  \eqref{basic equ'} and  the other use this but with $f$ replaced  by $ -\Delta^{\Sigma} f$;
in the last step, $\top$ denotes the projection onto the tangent bundle of $\Sigma$ and $F^{\star}$ denotes the pull-back of tensors corresponding to the map: $F:\Sigma \rightarrow S^{m+1} $.

For the term $A\left(\nabla^{\Sigma}\langle \nabla u,\frac{\partial}{\partial x_{\alpha}}  \rangle,\nabla^{\Sigma}\langle \nabla u,\frac{\partial}{\partial x_{\alpha}}  \rangle\right)$,
at $q$, we  have
\begin{align}
& \sum_{\alpha=1}^{m+2}A\left(\nabla^{\Sigma}\langle \nabla u,\frac{\partial}{\partial x_{\alpha}}  \rangle,\nabla^{\Sigma}\langle \nabla u,\frac{\partial}{\partial x_{\alpha}}  \rangle\right) \notag\\
=&\sum_{\alpha=1}^{m+2}\sum_{i=1}^{m} \mu_i \left(e_i \langle \nabla u,\frac{\partial}{\partial x_{\alpha}} \rangle\right)^2\notag\\
=& \sum_{\alpha=1}^{m+2}\sum_{i=1}^{m} \mu_i \left(\langle D^{\mathbb{R}^{m+2}}_{e_i}\nabla u, \frac{\partial}{\partial x_{\alpha}}\rangle \right)^2\notag\\
=&\sum_{\alpha=1}^{m+2}\sum_{i=1}^{m} \mu_i \left(\langle D_{e_i}\nabla u, \frac{\partial}{\partial x_{\alpha}}\rangle -x_{\alpha}\langle \nabla u,e_i \rangle \right)^2 \label{first derivate'}\\
=&\sum_{i=1}^{m}\mu_i \Big(  |D_{e_i} \nabla u|^2+ \langle \nabla^{\Sigma} f,e_i \rangle^2 \Big) \notag\\
=&A\Big(\big( D_{\mathbf{n}}\nabla u\big)^{\top},\big( D_{\mathbf{n}}\nabla u\big)^{\top}\Big)+\text{Trace}_{\Sigma} \Bigg (A\Big(\big( D_{()} \nabla u \big)^{\top} ,\big( D_{()} \nabla u \big)^{\top}\Big)\Bigg) \notag\\
&+A(\nabla^{\Sigma} f,\nabla^{\Sigma} f ).\notag
\end{align}
where in \eqref{first derivate'}, we use \eqref{first derivate};
in the last step,
$$\text{Trace}_{\Sigma} \Bigg (A\Big(\big( D_{()} \nabla u \big)^{\top} ,\big( D_{()} \nabla u \big)^{\top}\Big)\Bigg)=\sum_{i=1}^{m}
A\Big(\big( D_{\overline{e_i}} \nabla u \big)^{\top} ,\big( D_{\overline{e_i}} \nabla u \big)^{\top}\Big)$$ in any local orthonormal frame $\{\overline{e_l} \}_{l=1}^{m}$ on $\Sigma$.

Note the above two computations hold for any point on  $\Sigma$ since $q$ is arbitrary and all terms in the final result of the  computations are globally smooth. Hence, we can integrate them over $\Sigma$. Applying Stokes theorem  gives
 $$\int_{\Sigma} \textup{div}^{\Sigma} (Y)=0,$$
$$\int_{\Sigma} \langle\nabla^{\Sigma} f,\nabla^{\Sigma}\langle \nabla u,\mathbf{n} \rangle \rangle=-\int_{\Sigma} \langle \nabla u,\mathbf{n} \rangle \Delta^{\Sigma} f,$$
$$\int_{\Sigma} \langle\nabla^{\Sigma}(\Delta u-2\Delta^{\Sigma} f),\nabla^{\Sigma}\langle \nabla u,\mathbf{n} \rangle \rangle=-\int_{\Sigma} \langle \nabla u,\mathbf{n} \rangle \cdot \Delta^{\Sigma} (\Delta u-2\Delta^{\Sigma}f),$$
and (by combining like terms  )
\begin{align}
&-2\sum\limits_{\alpha=1}^{m+2} \int_{\Sigma} \left\langle \nabla^{\Sigma} \Big\langle \nabla\langle \nabla u,\frac{\partial}{\partial x_{\alpha}} \rangle, \mathbf{n}  \Big \rangle , \nabla^{\Sigma} \langle \nabla u,\frac{\partial}{\partial x_{\alpha}}  \rangle   \right\rangle \notag\\
&-\sum\limits_{\alpha=1}^{m+2}\int_{\Sigma} A\left(\nabla^{\Sigma}\langle \nabla u,\frac{\partial}{\partial x_{\alpha}}  \rangle,\nabla^{\Sigma}\langle \nabla u,\frac{\partial}{\partial x_{\alpha}}  \rangle\right) \notag
\end{align}
we continue from the previous page:
\begin{align}
=&-2\int_{\Sigma} A\left(\nabla^{\Sigma}(\Delta u-2\Delta^{\Sigma} f ),\nabla^{\Sigma} f\right)+2 \int_{\Sigma} \langle \nabla u,\mathbf{n} \rangle \cdot \Delta^{\Sigma} (\Delta u-2\Delta^{\Sigma}f) \notag\\
&+2(4-m)\int_{\Sigma} \langle \nabla u,\mathbf{n} \rangle \cdot \Delta^{\Sigma} f+(2m-5)\int_{\Sigma} A(\nabla^{\Sigma} f,\nabla^{\Sigma} f )\notag\\
&-5\int_{\Sigma}A\Big(\big( D_{\mathbf{n}}\nabla u\big)^{\top},\big( D_{\mathbf{n}}\nabla u\big)^{\top}\Big)+2\int_{\Sigma} (\Delta u-\Delta^{\Sigma} f )\cdot \langle A,F^{\star}\big(D^2u \big)\rangle \notag\\
&-\int_{\Sigma} \text{Trace}_{\Sigma} \Bigg (A\Big(\big( D_{()} \nabla u \big)^{\top} ,\big( D_{()} \nabla u \big)^{\top}\Big)\Bigg).\label{final result 1}
\end{align}
Summing \eqref{Reilly application} over $\alpha$ from 1 to $m+2$ and then substituting  \eqref{final result 1}, \eqref{hession computation},\eqref{gradient computation}, \eqref{squared laplace} into this expression   give
\begin{align}
&-2\int_{\Sigma} A\left(\nabla^{\Sigma}(\Delta u-2\Delta^{\Sigma} f ),\nabla^{\Sigma} f\right)+2 \int_{\Sigma} \langle \nabla u,\mathbf{n} \rangle \cdot \Delta^{\Sigma} (\Delta u-2\Delta^{\Sigma}f) \notag\\
&+2(4-m)\int_{\Sigma} \langle \nabla u,\mathbf{n} \rangle \cdot \Delta^{\Sigma} f+(2m-5)\int_{\Sigma} A(\nabla^{\Sigma} f,\nabla^{\Sigma} f )\notag\\
&-5\int_{\Sigma}A\Big(\big( D_{\mathbf{n}}\nabla u\big)^{\top},\big( D_{\mathbf{n}}\nabla u\big)^{\top}\Big)+2\int_{\Sigma} (\Delta u-\Delta^{\Sigma} f )\cdot \langle A,F^{\star}\big(D^2u \big)\rangle \notag\\
&-\int_{\Sigma} \text{Trace}_{\Sigma} \Bigg (A\Big(\big( D_{()} \nabla u \big)^{\top} ,\big( D_{()} \nabla u \big)^{\top}\Big)\Bigg)\notag\\
=&\int_{\Omega_1}|D^3 u|^2+(m+4)\int_{\Omega_1}|D^2 u|^2 -(m^2-2m)\int_{\Omega_1}|\nabla u|^2\notag\\
&-\int_{\Omega_1}|\nabla \Delta u |^2-4\int_{\Omega_1}(\Delta u)^2 -2m\int_{\Omega_1}\langle\nabla \Delta u,\nabla u    \rangle .\label{final result 2}
\end{align}
Note by Proposition \ref{Reilly}, 
\begin{align}
    \int_{\Omega_1}|D^2 u|^2=& -2\int_{\Sigma} \langle \nabla^{\Sigma} \langle \nabla u,\mathbf{n} \rangle ,\nabla^{\Sigma} f \rangle  -\int_{\Sigma} A(\nabla^{\Sigma} f, \nabla^{\Sigma} f)\notag\\
    &-m\int_{\Omega_1}|\nabla u|^2 +\int_{\Omega_1}(\Delta u)^2\notag\\
    =& 2 \int_{\Sigma} \langle \nabla u,\mathbf{n} \rangle \cdot \Delta^{\Sigma}f-  \int_{\Sigma} A(\nabla^{\Sigma} f, \nabla^{\Sigma} f)\notag\\
    &-m\int_{\Omega_1}|\nabla u|^2 +\int_{\Omega_1}(\Delta u)^2.\notag\\    \label{Relly'}
    \end{align}
 The lemma follows by substituting \eqref{Relly'} into \eqref{final result 2} and simplifying it.
\end{proof}
The results of Lemma \ref{some diffential computations}, Proposition \ref{Reilly} and Lemma \ref{integral  computations}  also apply to $v$. Now, we prove Theorem \ref{computation}.
\begin{proof}[Proof of Theorem \ref{computation}]
Since $\mathbf{n}$ is  the inward-pointing unit normal vector field with respect to $\Omega_1$ and $A$ is  the corresponding  second fundamental form, $-\mathbf{n}$ is the inward-pointing unit normal vector field with respect to $\Omega_2$ and $-A$ is  the corresponding  second fundamental form. Then, applying  Lemma \ref{integral  computations} to $v$ and noting that $v|_{\Sigma}=f$, we obtain
\begin{align}
&2\int_{\Sigma} A\left(\nabla^{\Sigma}(\Delta v-2\Delta^{\Sigma} f ),\nabla^{\Sigma} f\right)-2 \int_{\Sigma} \langle \nabla v,\mathbf{n} \rangle \cdot \Delta^{\Sigma} (\Delta v-2\Delta^{\Sigma}f) \notag\\
&+4m\int_{\Sigma} \langle \nabla v,\mathbf{n} \rangle \cdot \Delta^{\Sigma} f-(3m-1)\int_{\Sigma} A(\nabla^{\Sigma} f,\nabla^{\Sigma} f )\notag\\
&+5\int_{\Sigma}A\Big(\big( D_{\mathbf{n}}\nabla v\big)^{\top},\big( D_{\mathbf{n}}\nabla v\big)^{\top}\Big)-2\int_{\Sigma} (\Delta v-\Delta^{\Sigma} f )\cdot \langle A, F^{\star}\big(D^2v \big)\rangle \notag\\
&+\int_{\Sigma} \text{Trace}_{\Sigma} \Bigg (A\Big(\big( D_{()} \nabla v \big)^{\top} ,\big( D_{()} \nabla v \big)^{\top}\Big)\Bigg)\notag\\
=&\int_{\Omega_2}|D^3 v|^2 -2m(m+1)\int_{\Omega_2}|\nabla v|^2 \notag\\
&-\int_{\Omega_2}|\nabla \Delta v |^2+m\int_{\Omega_2}(\Delta v)^2 -2m\int_{\Omega_2}\langle\nabla \Delta v,\nabla v    \rangle .\label{final result 4}
\end{align}
By adding the result of Lemma \ref{integral  computations} and the inequality \eqref{final result 4}, the theorem follows.
\end{proof}

\section{Proof of Theorem \ref{main}}\label{main proof}
So far, we are now ready for the main theorem. The proof will follow from  Proposition \ref{function in tubular} and Theorem \ref{computation}. In this section,  we continue to  use the notations of Theorem \ref{computation} : $\mathbf{n}$ is the inward-pointing unit normal vector field with respect to $\Omega_1$, $A$ is the corresponding second fundamental form, and the case of $\Omega_2$ differs from that $\Omega_1$ only by a sign. Furthermore, we only need to prove the case that $|A|^2_{\max} >m$ because if $|A|^2_{\max} \leq m$, then $\Sigma$ is a totally geodesic sphere or a Clifford torus by the works of \cite{Simons1968minimal} and by either \cite{CherndoCarmominimal} or \cite{Lawson1969rigidity}, in which $\lambda_1=m$. In the following proof, we always assume that $|A|^2_{\max} >m$.

First, let's review  Choi and Wang's work \cite{chw1983first}.
\begin{thm}[\text{\cite{chw1983first}}]\label{Choi-Wang's work}
    Under the assumptions of Theorem \ref{main}, $$\lambda_1 \geq \frac{m}{2}.$$
\end{thm}
\begin{proof}
    Let $f$ be an eigenfunction with eigenvalue $\lambda_1$, i.e.,
\begin{align}
  \Delta^{\Sigma}f=-\lambda_1 f.\notag
\end{align}
In the regions $\Omega_1 $ and $\Omega_2$, we, respectively, solve the following Dirichlet problem:
\begin{equation*}
  \left\{\begin{array}{c}
     \Delta u= 0, \\
 u|_{\Sigma} =f.
  \end{array}
\right.
\end{equation*}
and  
\begin{equation*}
  \left\{\begin{array}{c}
     \Delta v= 0, \\
 u|_{\Sigma} =f.
  \end{array}
\right.
\end{equation*}
 By Proposition \ref{Reilly} or \eqref{Relly'}, we have 
\begin{align}
 \int_{\Omega_1} |D^2 u|^2= -2\lambda_1 \int_{\Sigma} f \langle \nabla u,\mathbf{n}\rangle-\int_{\Sigma} A(\nabla^{\Sigma},\nabla^{\Sigma} f )-m\int_{\Omega_1} |\nabla u|^2 \notag
\end{align}
and 
\begin{align}
 \int_{\Omega_2} |D^2 v|^2= 2\lambda_1 \int_{\Sigma} f \langle \nabla v,\mathbf{n}\rangle +\int_{\Sigma} A(\nabla^{\Sigma},\nabla^{\Sigma} f )-m\int_{\Omega_2} |\nabla v|^2.\notag
\end{align}
Using integration by parts gives 
\begin{align}
   -\int_{\Omega_1} f \langle \nabla u,\mathbf{n}\rangle=\int_{\Omega_1}|\nabla u|^2\label{basic integral u}
   \end{align}
and 
\begin{align}
  \int_{\Omega_2} f \langle \nabla v,\mathbf{n}\rangle=\int_{\Omega_2}|\nabla v|^2,\label{basic integral v}
\end{align}
where we note that $\mathbf{n}$ is the outward-pointing unit normal vector field with respect to $\Omega_2$.

Then, combining the above results, we obtain
\begin{align}
\int_{\Omega_1} |D^2 u|^2+ \int_{\Omega_2} |D^2 v|^2=(2\lambda_1 -m) \Big (\int_{\Omega_1} |\nabla u|^2+\int_{\Omega_2} |\nabla v|^2\Big).\label{basic inegral}
\end{align}
In particular,
\begin{align}
    \lambda_1 \geq \frac{m}{2}.\notag
\end{align}
 \end{proof}

Since 
\begin{align}
   D^2 u (\mathbf{n},\mathbf{n})= D^2 v(\mathbf{n},\mathbf{n})=-\Delta^{\Sigma} f=\lambda_1 f 
\end{align}
by the equality \eqref{basic equ} and by $\Delta u=0, \Delta v=0$, 
the integrals 
$$\int_{\Omega_1} |D^2 u|^2+ \int_{\Omega_2} |D^2 v|^2$$ is greater than 0. This means the inequality above is  actually  strict. Hence, in the rest of the section, we are going to estimate the lower bound of 
$$\int_{\Omega_1} |D^2 u|^2+ \int_{\Omega_2} |D^2 v|^2.$$
In the following discussion, we always assume that 
$$ \Delta^{\Sigma} f=-\lambda_1 f $$ and 
in $\Omega_1$,\begin{equation*}
  \left\{\begin{array}{c}
     \Delta u= 0, \\
 u|_{\Sigma} =f;
  \end{array}
\right.
\end{equation*}
 in $\Omega_2$
\begin{equation*}
  \left\{\begin{array}{c}
     \Delta v= 0, \\
 u|_{\Sigma} =f.
  \end{array}
\right.
\end{equation*}
We put the information of $f,u,v$  into Theorem \ref{computation} and note that \eqref{basic integral u} and \eqref{basic integral v}. This  directly gives 
\begin{lem}\label{forumla}
\begin{align}
&5\int_{\Sigma}A\Big(\big( D_{\mathbf{n}}\nabla v\big)^{\top},\big( D_{\mathbf{n}}\nabla v\big)^{\top}\Big)- 5\int_{\Sigma}A\Big(\big( D_{\mathbf{n}}\nabla u\big)^{\top},\big( D_{\mathbf{n}}\nabla u\big)^{\top}\Big) \notag\\
&+2\lambda_1 \int_{\Sigma} f\Big\langle A,F^{\star}\big(D^2u \big)-F^{\star}\big(D^2v \big) \Big\rangle+ \int_{\Sigma} \textup{Trace}_{\Sigma} \Bigg (A\Big(\big( D_{()} \nabla v \big)^{\top} ,\big( D_{()} \nabla v \big)^{\top}\Big)\Bigg)\notag\\
&-\int_{\Sigma} \textup{Trace}_{\Sigma} \Bigg (A\Big(\big(D_{()} \nabla u \big)^{\top} ,\big( D_{()} \nabla u \big)^{\top}\Big)\Bigg)\notag\\
=&2(2m\lambda_1-2\lambda_1 ^2-m^2-m)\int_{\Omega_1}|\nabla u|^2+ 2(2m\lambda_1-2\lambda_1 ^2-m^2-m)\int_{\Omega_2}|\nabla v|^2\notag\\
&+\int_{\Omega_1}|D^3 u|^2 + \int_{\Omega_2}|D^3 v|^2.\notag
\end{align}
\end{lem}

Applying Proposition \ref{function in tubular} to $|D^2 u|^2$ in $\Omega_1$ and to $|D^2 v|^2$ in $\Omega_2$  and combining the result with Lemma \ref{forumla} yield
\begin{lem} \label{final foruma 1}
 For any $\epsilon>0$, we have
\begin{align}
&\int_{\Sigma} \left (|D^2 u|^2+ |D^2 v|^2\right) \notag\\
\leq& \left[(2\lambda_1-m)\cdot\left(2\sqrt{\frac{m}{m-1}}|A|_{\max}+\frac{1}{\epsilon} \right)+2\epsilon (2\lambda_1^2-2 m\lambda_1 + m^2+m)\right]\notag\\
&\times \left(\int_{\Omega_1}|\nabla u|^2+ \int_{\Omega_2}|\nabla v|^2 \right)\notag\\    
&+5\epsilon  \bigintssss_{\Sigma}\bigg(A\Big(\big( D_{\mathbf{n}}\nabla v\big)^{\top},\big( D_{\mathbf{n}}\nabla v\big)^{\top}\Big)-A\Big(\big( D_{\mathbf{n}}\nabla u\big)^{\top},\big( D_{\mathbf{n}}\nabla u\big)^{\top}\Big)\bigg)\notag\\
&+\epsilon \bigintsss_{\Sigma} \Bigg[\textup{Trace}_{\Sigma} \Bigg (A\Big(\big( D_{()} \nabla v \big)^{\top} ,\big( D_{()} \nabla v \big)^{\top}\Big)\Bigg)\notag\\
& \qquad\qquad- \textup{Trace}_{\Sigma} \Bigg (A\Big(\big(D_{()} \nabla u \big)^{\top} ,\big( D_{()} \nabla u \big)^{\top}\Big)\Bigg)\Bigg]\notag\\
&+2\epsilon \lambda_1 \int_{\Sigma} f\Big\langle A,F^{\star}\big(D^2u \big)-F^{\star}\big(D^2v \big) \Big\rangle.\notag
\end{align}
\end{lem}
\begin{proof}
Applying Proposition \ref{function in tubular} to $|D^2 u|^2$ in $\Omega_1$ and to $|D^2 v|^2$ in $\Omega_2$ yields
\begin{align}
 \int_{\Sigma} |D^2 u|^2 \leq 2\sqrt{\frac{m}{m-1}}\cdot|A|_{\max}\int_{\Omega_1}|D^2u|^2 +\int_{\Omega_1} |\nabla |D^2 u|^2 | \label{corollary application u}
\end{align}
and 
\begin{align}
 \int_{\Sigma} |D^2 v|^2 \leq 2\sqrt{\frac{m}{m-1}} \cdot|A|_{\max}\int_{\Omega_2}|D^2v|^2 +\int_{\Omega_2} |\nabla |D^2 v|^2 |.\label{corollary application v}
\end{align}

 Choose a local orthonormal frame $\{\tilde{e_i}\}_{i=1}^{m+1}$.
Then 
\begin{align}
|\nabla |D^2 u|^2|^2=&4\sum\limits_{j=1}^{m+1} \Big(\sum\limits_{1\leq i,k \leq m+1} D^2u (\tilde{e_i},\tilde{e_k} )^2 D^3 u(\tilde{e_i},\tilde{e_k},\tilde{e_j})\Big)^2\notag\\
\leq & 4 |D^2 u|^2\cdot |D^3 u|^2.\notag
\end{align}
where we use the Cauchy-Schwarz inequality, and  the case of $v$ is the same.

That is,
\begin{align}
|\nabla |D^2 u|^2| \leq 2 |D^2u| \cdot |D^3 u| \label{D^3 u estimate}
\end{align}
and 
\begin{align}
|\nabla |D^2 v|^2| \leq 2 |D^2v| \cdot |D^3 v| \label{D^3 v estimate}.
\end{align}
Put \eqref{D^3 u estimate} and \eqref{D^3 v estimate} into \eqref{corollary application u} and \eqref{corollary application v}, respectively, and then add \eqref{corollary application u} and \eqref{corollary application v}.
We obtain,  for any $\epsilon>0$,
\begin{align}
\int_{\Sigma} \left (|D^2 u|^2+ |D^2 v|^2\right) \leq &2\sqrt{\frac{m}{m-1}}\cdot|A|_{\max}\left(\int_{\Omega_1}|D^2u|^2+ \int_{\Omega_2}|D^2v|^2 \right)\notag\\ 
&+2\int_{\Omega_1} |D^2u| \cdot |D^3 u| +2 \int_{\Omega_2} |D^2v| \cdot |D^3 v|\notag\\
\leq & \left(2\sqrt{\frac{m}{m-1}}\cdot|A|_{\max}+\frac{1}{\epsilon} \right)\left(\int_{\Omega_1}|D^2u|^2+ \int_{\Omega_2}|D^2v|^2 \right)\notag\\
&+ \epsilon \left(\int_{\Omega_1} |D^3u|^2 +\int_{\Omega_2} |D^3v|^2     \right),\label{form forumla}
\end{align}
where use the inequality 
$$2ab=2\cdot \frac{a}{\sqrt{\epsilon}} \cdot \sqrt{\epsilon} b \leq \frac{a^2}{\epsilon} +\epsilon b^2 .$$
Then the lemma follows from Lemma \ref{forumla} and \eqref{basic inegral}.
\end{proof}
Before estimating the last three terms in the inequality, we first list the basic information $u$ and $v$ on the boundary for subsequent use.
Note that in a local orthonormal frame $\{e_i\}_{i=1}^{m}$ on $\Sigma$,
\begin{align}
D^2 u(e_i,e_j)=\big( D^{\Sigma}\big)^2 f(e_i,e_j)-A(e_i,e_j) \langle \nabla u,\mathbf{n} \rangle, \notag\\
D^2 v(e_i,e_j)=\big( D^{\Sigma}\big)^2 f(e_i,e_j)-A(e_i,e_j) \langle \nabla v,\mathbf{n} \rangle,\notag\\
D^2 u(\mathbf{n},e_i) =D^2 u(e_i,\mathbf{n})=A(\nabla^{\Sigma} f,e_i)+ \langle \nabla^{\Sigma}\langle\nabla u,\mathbf{n} \rangle,e_i  \rangle, \notag\\
D^2 v(\mathbf{n},e_i) =D^2 v(e_i,\mathbf{n})=A(\nabla^{\Sigma} f,e_i)+ \langle \nabla^{\Sigma}\langle\nabla v,\mathbf{n} \rangle,e_i  \rangle,\notag\\
D^2 u(\mathbf{n},\mathbf{n})=D^2 v(\mathbf{n},\mathbf{n})=-\Delta^{\Sigma} f=\lambda_1 f,\label{boundary information}
\end{align}
which implies
\begin{align}
\Big|F^{\star}\big(D^2u \big)-F^{\star}\big(D^2v \big) \Big|^2=&\sum\limits_{1 \leq i,j\leq m} \big(D^2 u(e_i,e_j)-D^2 v (e_i,e_j)\big)^2 \notag\\
=&|A|^2 \big( \langle \nabla v ,\mathbf{n}  \rangle-\langle \nabla u,\mathbf{n} \rangle  \big)^2, \label{boundary information 1}\\
\Big|\big (D_{\mathbf{n}} \nabla u)^{\top}-\big(D_{\mathbf{n}}\nabla v)^{\top}\Big|^2=&\sum\limits_{i=1}^{m} \big(D^2u(e_i,\mathbf{n})-D^2 v(e_i,\mathbf{n})\big)^2\notag\\
=&\big|\nabla^{\Sigma} \langle \nabla u,\mathbf{n}\rangle
-\nabla^{\Sigma} \langle \nabla v,\mathbf{n} \rangle\big|^2,\label{boundary information 2}
\end{align}
\begin{align}
|D^2 u|^2+|D^2 v|^2=&\Big (D^2 u(\mathbf{n},\mathbf{n})  \Big)^2+ \Big (D^2 v(\mathbf{n},\mathbf{n})  \Big)^2+2 \Big|\big (D_{\mathbf{n}} \nabla u)^{\top}\Big|^2+2\Big|\big (D_{\mathbf{n}} \nabla v)^{\top}\Big|^2\notag\\
&+\Big|F^{\star}\big(D^2u \big)\Big|^2+\Big|F^{\star}\big(D^2v \big) \Big|^2\notag\\
=&2\lambda_1 ^2 f^2 + \Big|\big( D_{\mathbf{n}}\nabla u\big)^{\top}+\big( D_{\mathbf{n}}\nabla v\big)^{\top}\Big|^2 +\big|\nabla^{\Sigma} \langle \nabla u,\mathbf{n}\rangle
-\nabla^{\Sigma} \langle \nabla v,\mathbf{n} \rangle\big|^2\notag\\
&+\frac{1}{2}\Big|F^{\star}\big(D^2u \big)+F^{\star}\big(D^2v \big) \Big|^2
+\frac{1}{2}\cdot|A|^2 \big( \langle \nabla v ,\mathbf{n}  \rangle-\langle \nabla u,\mathbf{n} \rangle  \big)^2.\label{boundary information 3}
\end{align}
Now, let's estimate the last three terms. 

First, for the term $$5\epsilon A\Big(\big( D_{\mathbf{n}}\nabla v\big)^{\top},\big( D_{\mathbf{n}}\nabla v\big)^{\top}\Big)-5\epsilon A\Big(\big( D_{\mathbf{n}}\nabla u\big)^{\top},\big( D_{\mathbf{n}}\nabla u\big)^{\top}\Big),$$ we obtain
\begin{lem}\label{D^2u,v(n,)}
   For any $\epsilon>0$, we have 
   \begin{align}
   & 5\epsilon A\Big(\big( D_{\mathbf{n}}\nabla v\big)^{\top},\big( D_{\mathbf{n}}\nabla v\big)^{\top}\Big)-5\epsilon A\Big(\big( D_{\mathbf{n}}\nabla u\big)^{\top},\big( D_{\mathbf{n}}\nabla u\big)^{\top}\Big)\notag\\
   \leq & \Big|\big (D_{\mathbf{n}} \nabla u)^{\top}+\big(D_{\mathbf{n}}\nabla v)^{\top}\Big|^2+ \frac{25\epsilon^2}{4}\cdot\frac{m-1}{m}|A|^2\cdot\big|\nabla^{\Sigma} \langle \nabla u,\mathbf{n}\rangle
-\nabla^{\Sigma} \langle \nabla v,\mathbf{n} \rangle\big|^2.\notag
\end{align}
\end{lem}    
   \begin{proof}
Fix $p\in \Sigma$ and choose a local orthonormal frame $\{e_i\}_{i=1}^{m}$ near $p$ on $\Sigma$ such that 
$$A(e_i,e_j)=\delta_{ij} \mu_i, 1\leq i,j\leq m.  $$
Then, at $p$, by \eqref{boundary information 2} and 
$$ |A|^2=\sum\limits_{i=1}^{m} \mu_i^2 \geq \mu_i^2 +\frac{\left(\sum\limits_{j\neq i} \mu_j \right)^2}{m-1}=\frac{m}{m-1}\mu_i^2,$$
we have
\begin{align}
   & 5\epsilon A\Big(\big( D_{\mathbf{n}}\nabla v\big)^{\top},\big( D_{\mathbf{n}}\nabla v\big)^{\top}\Big)-5\epsilon A\Big(\big( D_{\mathbf{n}}\nabla u\big)^{\top},\big( D_{\mathbf{n}}\nabla u\big)^{\top}\Big)\notag\\
   =&5\epsilon \sum\limits_{i=1}^{m}\mu_i \Big( D^2 v (e_i,\mathbf{n} ) \Big)^2 -5\epsilon \sum\limits_{i=1}^{m}\mu_i \Big( D^2 u (e_i,\mathbf{n} ) \Big)^2\notag\\
   =&5\epsilon \sum\limits_{i=1}^{m}\mu_i \Big( D^2 v (e_i,\mathbf{n} )+  D^2 u (e_i,\mathbf{n} )\Big)\cdot \Big( D^2 v (e_i,\mathbf{n} )- D^2 u (e_i,\mathbf{n} )\Big)\notag\\
   \leq & \sum\limits_{i=1}^{m} \Big( D^2 v (e_i,\mathbf{n} )+  D^2 u (e_i,\mathbf{n} )\Big)^2 +\frac{25\epsilon^2}{4}\sum\limits_{i=1}^{m} 
   \mu_i ^2\Big( D^2 v (e_i,\mathbf{n} )- D^2 u (e_i,\mathbf{n} )\Big)^2 \notag\\
   \leq &\Big|\big (D_{\mathbf{n}} \nabla u)^{\top}+\big(D_{\mathbf{n}}\nabla v)^{\top}\Big|^2+ \frac{25\epsilon^2}{4}\cdot\frac{m-1}{m}|A|^2\sum\limits_{i=1}^{m}
 \Big( D^2 u (e_i,\mathbf{n})-D^2 v(e_i,\mathbf{n} )\Big)^2\notag\\  
 =&\Big|\big (D_{\mathbf{n}} \nabla u)^{\top}+\big(D_{\mathbf{n}}\nabla v)^{\top}\Big|^2+ \frac{25\epsilon^2}{4}\cdot\frac{m-1}{m}|A|^2\cdot\big|\nabla^{\Sigma} \langle \nabla u,\mathbf{n}\rangle
-\nabla^{\Sigma} \langle \nabla v,\mathbf{n} \rangle\big|^2,\notag
   \end{align}
where we use the inequality 
$$ 5\epsilon ab =\sqrt{2}a \cdot \frac{5\epsilon}{\sqrt{2}} b\leq a^2 +\frac{25\epsilon^2}{4}b^2,$$
and this holds for any point in $\Sigma$.

   This completes the proof of this lemma.
    \end{proof}

Second, for the remaining term 
\begin{align}
&\epsilon \cdot \text{Trace}_{\Sigma} \Bigg (A\Big(\big( D_{()} \nabla v \big)^{\top} ,\big( D_{()} \nabla v \big)^{\top}\Big)\Bigg)- \epsilon \cdot \text{Trace}_{\Sigma} \Bigg (A\Big(\big(D_{()} \nabla u \big)^{\top} ,\big( D_{()} \nabla u \big)^{\top}\Big)\Bigg)\notag\\
&+2\epsilon \lambda_1  f\Big\langle A,F^{\star}\big(D^2 u \big)-F^{\star}\big(D^2 v \big) \Big\rangle,\notag
\end{align}
we obtain
\begin{lem}\label{D^2u,v(,)}
    For any absolute constant $\epsilon>0$ and  any  positive function $\beta$ defined on $\Sigma$, we have 
    \begin{align}
&\epsilon \cdot \textup{Trace}_{\Sigma} \Bigg (A\Big(\big( D_{()} \nabla v \big)^{\top} ,\big( D_{()} \nabla v \big)^{\top}\Big)\Bigg)- \epsilon \cdot \textup{Trace}_{\Sigma} \Bigg (A\Big(\big(D_{()} \nabla u \big)^{\top} ,\big( D_{()} \nabla u \big)^{\top}\Big)\Bigg)\notag\\
&+2\epsilon \lambda_1  f\Big\langle A,F^{\star}\big(D^2 u \big)-F^{\star}\big(D^2 v \big) \Big\rangle\notag\\
\leq & \frac{1}{2} \Big|F^{\star}\big(D^2 u \big)+F^{\star}\big(D^2 v \big) \Big|^2-\frac{2\lambda_1^2}{m} f^2 +\frac{\epsilon^2(m^2-3m+3)}{2m(m-1)}|A|^4\big( \langle \nabla v ,\mathbf{n}  \rangle-\langle \nabla u,\mathbf{n} \rangle  \big)^2 \notag\\
&+ \epsilon \beta |A|^2 \big( \langle \nabla v ,\mathbf{n}  \rangle-\langle \nabla u,\mathbf{n} \rangle  \big)^2+ \left(\frac{m+1}{m}\right)^2 \frac{\epsilon\lambda_1 ^2}{\beta} \cdot|A|^2 f^2.\notag
    \end{align}
\end{lem}
 \begin{proof}
Fix $p\in \Sigma$ and choose a local orthonormal frame $\{e_i\}_{i=1}^{m}$ near $p$ on $\Sigma$ such that 
$$A(e_i,e_j)=\delta_{ij} \mu_i, 1\leq i,j\leq m.  $$
 Then at $p$, by \eqref{boundary information}, we  obtain
\begin{align}
&\epsilon \cdot \text{Trace}_{\Sigma} \Bigg (A\Big(\big( D_{()} \nabla v \big)^{\top} ,\big( D_{()} \nabla v \big)^{\top}\Big)\Bigg)- \epsilon \cdot \text{Trace}_{\Sigma} \Bigg (A\Big(\big(D_{()} \nabla u \big)^{\top} ,\big( D_{()} \nabla u \big)^{\top}\Big)\Bigg)\notag\\
&+2\epsilon \lambda_1  f\Big\langle A,F^{\star}\big(D^2 u \big)-F^{\star}\big(D^2 v \big) \Big\rangle\notag\\
=&\epsilon\sum_{1\leq i,j \leq m} \mu_i\Big(D^2 v(e_i,e_j) \Big)^2-\epsilon\sum_{1\leq i,j \leq m} \mu_i\Big(D^2 u(e_i,e_j) \Big)^2\notag\\
&+ 2\epsilon \lambda_1 f \sum\limits_{i=1}^{m}\mu_i \Big (D^2 u (e_i,e_i)-D^2 v(e_i,e_i)\Big)\notag\\
=&\epsilon \sum_{1\leq i,j \leq m}\mu_i \Big(D^2 v(e_i,e_j)+D^2 u(e_i,e_j) +\frac{2\lambda_1}{m} \delta_{ij} f\Big)\cdot \Big(D^2 v(e_i,e_j)-D^2 u(e_i,e_j)\Big)
\notag\\
&+ 2 \left(1+\frac{1}{m}\right) \epsilon \lambda_1 |A|^2 f\big( \langle \nabla v ,\mathbf{n}  \rangle-\langle \nabla u,\mathbf{n} \rangle  \big)\notag \\ 
\leq & \frac{1}{2}\sum_{1\leq i,j \leq m}  \Big(D^2 v(e_i,e_j)+D^2 u(e_i,e_j)+\frac{2\lambda_1}{m} \delta_{ij} f\Big)^2 \notag\\
&+\frac{\epsilon^2}{2}\sum\limits_{1\leq i,j \leq m}
\mu_i ^2  \Big(D^2 v(e_i,e_j)-D^2 u(e_i,e_j)\Big)^2\notag\\
&+ \epsilon \beta |A|^2 \big( \langle \nabla v ,\mathbf{n}  \rangle-\langle \nabla u,\mathbf{n} \rangle  \big)^2+ \left(1+\frac{1}{m}\right)^2 \frac{\epsilon\lambda_1 ^2}{\beta} \cdot|A|^2 f^2\notag\\
=&\frac{1}{2} \Big|F^{\star}\big(D^2 u \big)+F^{\star}\big(D^2 v \big) \Big|^2-\frac{2\lambda_1^2}{m} f^2 +\frac{\epsilon^2}{2}\big( \langle \nabla v ,\mathbf{n}  \rangle-\langle \nabla u,\mathbf{n} \rangle  \big)^2\sum\limits_{i=1}^{m}\mu_i ^4\notag\\
&+ \epsilon \beta |A|^2 \big( \langle \nabla v ,\mathbf{n}  \rangle-\langle \nabla u,\mathbf{n} \rangle  \big)^2+ \left(1+\frac{1}{m}\right)^2 \frac{\epsilon\lambda_1 ^2}{\beta} \cdot|A|^2 f^2,\label{integral estimate i,j}
\end{align}
 where $\beta>0$ is an arbitrary positive function defined on $\Sigma$, and we use inequalities:
 $$\epsilon ab \leq \frac{1}{2} a^2+\frac{\epsilon^2}{2} b^2   $$ and 
 $$2 \left(1+\frac{1}{m}\right)  \lambda_1 ab \leq \left(1+\frac{1}{m}\right)^2\frac{\lambda_1^2}{\beta}a^2+\beta b^2.$$
 By Lagrange  multiplier theory, we can find 
 \begin{align}
\sum\limits_{i=1}^{m}\mu_i ^4 \leq \frac{m^2-3m+3}{m(m-1)}|A|^4\notag
 \end{align}
 under the constraints:
 $$ \sum\limits_{i=1}^{m}\mu_i=0,\sum\limits_{i=1}^{m}\mu_i ^2 =|A|^2.$$
Substituting this result into \eqref{integral estimate i,j}, we get this lemma.
\end{proof}

Plugging the results of Lemma \ref{D^2u,v(n,)} and  Lemma \ref{D^2u,v(,)}, and the equality \eqref{boundary information 3} into Lemma \ref{final foruma 1}   and sorting it out immediately yield
\begin{lem}\label{final formula 2}
For any absolute constant   $\epsilon>0$ and any  positive function $\beta$ defined on $\Sigma$,
\begin{align}
&\left[(2\lambda_1-m)\cdot\left(2\sqrt{\frac{m}{m-1}}\cdot|A|_{\max}+\frac{1}{\epsilon} \right)+2\epsilon (2\lambda_1^2-2 m\lambda_1 + m^2+m)\right] \notag\\
&\quad \times \left(\int_{\Omega_1}|\nabla u|^2+ \int_{\Omega_2}|\nabla v|^2 \right)\notag\\
\geq &  \frac{m+1}{m} \lambda_1 ^2 \bigintsss_{\Sigma}\Bigg( 2-\frac{m+1}{m}\frac{\epsilon}{\beta}\cdot |A|^2\Bigg)  f^2\notag\\
&+ \bigintsss_{\Sigma} \left(1-\frac{25\epsilon^2}{4}\cdot \frac{m-1}{m}|A|^2\right)\cdot \big|\nabla^{\Sigma} \langle \nabla u,\mathbf{n}\rangle
-\nabla^{\Sigma} \langle \nabla v,\mathbf{n} \rangle\big|^2\notag\\
&+ \bigintsss_{\Sigma}\left(\frac{1}{2}-\epsilon \beta-\frac{m^2-3m+3}{2m(m-1)}\epsilon^2 |A|^2\right)\cdot|A|^2 \big( \langle \nabla v ,\mathbf{n}  \rangle-\langle \nabla u,\mathbf{n} \rangle  \big)^2.\notag
\end{align}
\end{lem}
 Impose a restriction on $\epsilon$ and choose appropriate $\beta$ (depending on $\epsilon$)  to obtain a further bound on Lemma \ref{final formula 2}. This gives 
 \begin{lem}\label{final formula 5}
 For any $\epsilon>0$, if 
   \begin{align}
\epsilon <\frac{2}{5} \sqrt{\frac{m}{m-1}}\cdot\frac{1}{|A|_{\max}}\ ,\label{restriction epsion}   
\end{align}
 then  
 \begin{align}
    &\left[(2\lambda_1-m)\cdot\left(  2\sqrt{\frac{m}{m-1}}\cdot|A|_{\max}+\frac{1}{\epsilon} \right)+2\epsilon (2\lambda_1^2-2 m\lambda_1 + m^2+m)\right] \notag\\
& \times \left(\int_{\Omega_1}|\nabla u|^2+ \int_{\Omega_2}|\nabla v|^2 \right)\notag\\
\geq &2\sqrt{\frac{m+1}{m}}\cdot\lambda_1  \notag\\
&\times\bigintsss_{\Sigma}\Bigg\{   \Bigg[ \sqrt{2\lambda_1\cdot\left(1-\frac{25\epsilon^2}{4}\cdot \frac{m-1}{m}|A|_{\max}^2\right)+|A|^2 \left(1-\frac{m^2-3m+3}{m(m-1)}\epsilon^2 |A|^2\right)}\notag\\
 &\qquad\qquad-\sqrt{\frac{m+1}{m}} \cdot\epsilon |A|^2\Bigg]\cdot\Big|f \big(\langle \nabla u ,\mathbf{n}  \rangle-\langle \nabla v,\mathbf{n} \rangle\big)  \Big|\Bigg\}.\notag
\end{align} 
\end{lem}
 \begin{proof}
$\Delta u=0, \Delta v=0$ and the divergence theorem give
$$\int_{\Sigma} \langle \nabla u ,\mathbf{n}\rangle =\int_{\Sigma} \langle \nabla v,\mathbf{n} \rangle=0.$$
Consequently, by  definition of $\lambda_1$, we have 
$$\int_{\Sigma} \big|\nabla^{\Sigma} \langle \nabla u,\mathbf{n}\rangle
-\nabla^{\Sigma} \langle \nabla v,\mathbf{n} \rangle\big|^2\geq \lambda_1\int_{\Sigma}\big( \langle \nabla u ,\mathbf{n}  \rangle-\langle \nabla v,\mathbf{n} \rangle  \big)^2$$ 

Substituting this result into Lemma \ref{final formula 2}, we get
\begin{align}
&\left[(2\lambda_1-m)\cdot\left(2\sqrt{\frac{m}{m-1}}\cdot|A|_{\max}+\frac{1}{\epsilon} \right)+2\epsilon (2\lambda_1^2-2 m\lambda_1 + m^2+m)\right] \notag\\
& \times \left(\int_{\Omega_1}|\nabla u|^2+ \int_{\Omega_2}|\nabla v|^2 \right)\notag\\
\geq &  \frac{m+1}{m}\lambda_1 ^2 \bigintsss_{\Sigma}\Bigg( 2-\frac{m+1}{m}\frac{\epsilon}{\beta}\cdot |A|^2\Bigg)  f^2\notag\\
&+ \bigintss_{\Sigma} \Bigg\{\Bigg[\lambda_1\cdot\left(1-\frac{25\epsilon^2}{4}\cdot \frac{m-1}{m}|A|_{\max}^2\right)+\left(\frac{1}{2}-\epsilon \beta-\frac{m^2-3m+3}{2m(m-1)}\epsilon^2 |A|^2\right)\cdot|A|^2\Bigg] \notag\\
&\qquad \times \big( \langle \nabla u ,\mathbf{n}  \rangle-\langle \nabla v,\mathbf{n} \rangle  \big)^2 \  \Bigg\},\label{final formula 3}
\end{align}
where we use  $|A|^2\leq |A|_{\max} ^2 $.

For later estimates, we impose some restrictions on $\beta$: fix $\epsilon$, and for each point on $\Sigma$, 
$$ 2-\frac{m+1}{m}\frac{\epsilon}{\beta}\cdot |A|^2>0  $$ and 
$$\lambda_1\cdot\left(1-\frac{25\epsilon^2}{4}\cdot \frac{m-1}{m}|A|_{\max}^2\right)+\left(\frac{1}{2}-\epsilon \beta-\frac{(m^2-3m+3)}{2m(m-1)}\epsilon^2 |A|^2\right)\cdot|A|^2 >0.$$
i.e., for any point on $\Sigma$, 
\begin{align}
\beta >\frac{m+1}{2m}\epsilon |A|^2,\label{restriction beta 1}
\end{align}
and 
\begin{align}
\beta&<-\frac{m^2-3m+3}{2m(m-1)}\epsilon  |A|^2+\frac{1}{2\epsilon}+\frac{\lambda_1}{|A|^2\epsilon} \cdot\left(1-\frac{25\epsilon^2}{4}\cdot \frac{m-1}{m}|A|_{\max}^2\right) \notag\\
&=\left(\frac{1}{2}+\frac{\lambda_1}{|A|^2}\right)\frac{1}{\epsilon}
-\left(\frac{m^2-3m+3}{2m(m-1)}  |A|^2 +\frac{25}{4}\cdot\frac{m-1}{m} \lambda_1 \frac{|A|_{\max}^2}{|A|^2}  \right)\epsilon, \ |A|^2\neq 0.  \label{restriction beta 2}
\end{align}
Assume that there exists $\beta$ such that it satisfies \eqref{restriction beta 1} and \eqref{restriction beta 2} under \eqref{restriction epsion}. Then,
using $a^2+b^2 \geq 2|a|\cdot|b|$ for the integrand in \eqref{final formula 3}  yields 
\begin{align}
&\left[(2\lambda_1-m)\cdot\left( 2\sqrt{\frac{m}{m-1}}\cdot|A|_{\max}  +\frac{1}{\epsilon} \right)+2\epsilon (2\lambda_1^2-2 m\lambda_1 + m^2+m)\right] \notag\\
& \times \left(\int_{\Omega_1}|\nabla u|^2+ \int_{\Omega_2}|\nabla v|^2 \right)\notag\\
\geq &2\sqrt{\frac{m+1}{m}}\cdot\lambda_1  \bigintsss_{\Sigma}\Bigg\{ \Big|f \big(\langle \nabla u ,\mathbf{n}  \rangle-\langle \nabla v,\mathbf{n} \rangle\big)  \Big|\cdot\sqrt{ 2-\frac{m+1}{m}\frac{\epsilon}{\beta}\cdot |A|^2}\notag\\
&\quad \times\sqrt{\lambda_1\cdot\left(1-\frac{25\epsilon^2}{4}\cdot \frac{m-1}{m}|A|_{\max}^2\right)+\left(\frac{1}{2}-\epsilon \beta-\frac{m^2-3m+3}{2m(m-1)}\epsilon^2 |A|^2\right)\cdot|A|^2}\ \ \Bigg\}.\label{final formula 4}
\end{align}
Using $a^2+b^2\geq 2|a|\cdot|b|$ again, we obtain
 \begin{align}
 &\left(2-\frac{m+1}{m}\frac{\epsilon}{\beta}\cdot|A|^2\right)\notag\\
 &\times  \Bigg[\lambda_1\cdot\left(1-\frac{25\epsilon^2}{4}\cdot \frac{m-1}{m}|A|_{\max}^2\right)+\left(\frac{1}{2}-\epsilon \beta-\frac{m^2-3m+3}{2m(m-1)}\epsilon^2 |A|^2\right)\cdot|A|^2\Bigg]\notag\\
 =&2\lambda_1\cdot\left(1-\frac{25\epsilon^2}{4}\cdot \frac{m-1}{m}|A|_{\max}^2\right)+|A|^2 \left(1-\frac{m^2-3m+3}{m(m-1)}\epsilon^2 |A|^2\right)+\frac{m+1}{m}\epsilon^2 |A|^4\notag\\
 &-\frac{m+1}{m}\frac{\epsilon}{\beta}\cdot|A|^2 \Bigg[\lambda_1\left(1-\frac{25\epsilon^2}{4}\cdot \frac{m-1}{m}|A|_{\max}^2\right)+|A|^2 \left(\frac{1}{2}-\frac{m^2-3m+3}{2m(m-1)}\epsilon^2 |A|^2\right)\Bigg] \notag\\
 &-2\epsilon |A|^2 \beta\notag\\
 \leq & 2\lambda_1\cdot\left(1-\frac{25\epsilon^2}{4}\cdot \frac{m-1}{m}|A|_{\max}^2\right)+|A|^2 \left(1-\frac{m^2-3m+3}{m(m-1)}\epsilon^2 |A|^2\right)+\frac{m+1}{m}\epsilon^2 |A|^4\notag\\
 &-2\sqrt{2} \cdot\sqrt{\frac{m+1}{m}}\epsilon \cdot |A|^2\notag\\
 &\times \sqrt{\lambda_1\left(1-\frac{25\epsilon^2}{4}\cdot \frac{m-1}{m}|A|_{\max}^2\right)+|A|^2 \left(\frac{1}{2}-\frac{m^2-3m+3}{2m(m-1)}\epsilon^2 |A|^2\right)}\notag\\
 =& \Bigg[ \sqrt{2\lambda_1\cdot\left(1-\frac{25\epsilon^2}{4}\cdot \frac{m-1}{m}|A|_{\max}^2\right)+|A|^2 \left(1-\frac{m^2-3m+3}{m(m-1)}\epsilon^2 |A|^2\right)}\notag\\
 &\quad -\sqrt{\frac{m+1}{m}} \cdot\epsilon |A|^2\Bigg]^2,\label{final formula 4'}
 \end{align}
where the equality holds if and only if 
$$\beta^2=\frac{m+1}{2m}\Bigg[\lambda_1\left(1-\frac{25\epsilon^2}{4}\cdot \frac{m-1}{m}|A|_{\max}^2\right)+|A|^2 \left(\frac{1}{2}-\frac{m^2-3m+3}{2m(m-1)}\epsilon^2 |A|^2\right)\Bigg]$$
 holds when  $|A|^2\neq 0$.

 Now, we need to verify that under the range of $\epsilon$ \eqref{restriction epsion}, this function satisfies \eqref{restriction beta 1} and \eqref{restriction beta 2}.
 Combining  this function  with \eqref{restriction beta 1} and \eqref{restriction beta 2}, we deduce that the existence $\beta$ is equivalent to the following inequality holding for any point on $\Sigma$:
 $$\lambda_1\left(1-\frac{25\epsilon^2}{4}\cdot \frac{m-1}{m}|A|_{\max}^2\right)+|A|^2 \left(\frac{1}{2}-\frac{m^2-3m+3}{2m(m-1)}\epsilon^2 |A|^2\right)> \frac{m+1}{2m}\epsilon^2 |A|^4,$$
that is,
$$\lambda_1\left(1-\frac{25\epsilon^2}{4}\cdot \frac{m-1}{m}|A|_{\max}^2\right) > \frac{|A|^2}{2}\left(\frac{2m^2-3m+2}{m(m-1)}|A|^2\epsilon^2-1 \right).$$
This is also  equivalent to the following inequality holding:
$$\lambda_1\left(1-\frac{25\epsilon^2}{4}\cdot \frac{m-1}{m}|A|_{\max}^2\right) > \frac{|A|_{\max}^2}{2}\left(\frac{2m^2-3m+2}{m(m-1)}|A|_{\max}^2\epsilon^2-1 \right).$$
 i.e.,
 $$\epsilon^2<\frac{|A|_{\max}^2+2\lambda_1}{ \frac{2m^2-3m+2}{m(m-1)}|A|_{\max}^4 +\frac{25}{2}\frac{m-1}{m} \lambda_1 |A|_{\max}^2}.$$
Comparing it with \eqref{restriction epsion}, we find  that the range above  of $\epsilon$ is larger than  \eqref{restriction epsion}, which implies that we can choose $\beta$ to be this function. Also, from the above inequality of  $\epsilon$, the term inside the square bracket in the final result of \eqref{final formula 4'} is positive.

Fix $\epsilon$ and let  
$$ \beta^2=\frac{m+1}{2m}\Bigg[\lambda_1\left(1-\frac{25\epsilon^2}{4}\cdot \frac{m-1}{m}|A|_{\max}^2\right)+|A|^2 \left(\frac{1}{2}-\frac{m^2-3m+3}{2m(m-1)}\epsilon^2 |A|^2\right)\Bigg].$$
Combining \eqref{final formula 4} and \eqref{final formula 4'}, we obtain this lemma.
 \end{proof}
Finally, we need a lemma to get a lower bound for the term 
$$\sqrt{2\lambda_1\cdot\left(1-\frac{25\epsilon^2}{4}\cdot \frac{m-1}{m}|A|_{\max}^2\right)+|A|^2 \left(1-\frac{m^2-3m+3}{m(m-1)}\epsilon^2 |A|^2\right)}-\sqrt{\frac{m+1}{m}} \cdot\epsilon |A|^2$$ 
that depends only on $\epsilon,|A|_{\max}$ and $|A|_{\min}$. This requires the further restriction on $\epsilon$. Then, based on this result, we can obtain a concise estimate independent of the functions $u$ and $v$. The lemma is as follows.
\begin{lem}\label{final formula 6}
For any $\epsilon>0$, if 
\begin{align}
0<\epsilon \leq \frac{1}{2}\sqrt{\frac{m-1}{4m-3}}\cdot \frac{1}{|A|_{\max}}\ ,  \label{new restriction  epsion}
\end{align}
then
\begin{align}
&(2\lambda_1-m)\left(2\sqrt{\frac{m}{m-1}}\cdot|A|_{\max}+\frac{1}{\epsilon} \right)\notag\\
&+2\epsilon \Bigg[2\lambda_1^2+\left(\frac{m+1}{m}|A|^2_{\min}-2 m\right)\lambda_1 +m^2+m\Bigg] \notag\\
\geq & 2\sqrt{\frac{m+1}{m}}\cdot\lambda_1\sqrt{2\lambda_1+|A|_{\min}^2-\left(\frac{25(m-1)\lambda_1}{2m}|A|^2_{\max}+\frac{m^2-3m+3}{m(m-1)}|A|_{\min}^4 \right)\epsilon^2}\ .\notag
\end{align}    
\end{lem}
\begin{proof}
  Let $$\eta(y)=\sqrt{2\lambda_1\cdot\left(1-\frac{25\epsilon^2}{4}\cdot \frac{m-1}{m}|A|_{\max}^2\right)+y \left(1-\frac{m^2-3m+3}{m(m-1)}\epsilon^2 y\right)}-\sqrt{\frac{m+1}{m}} \cdot\epsilon y, $$ 
    where $y \in \big[|A|^2_{\min} , |A|^2_{\max}\big].$
    
    Our goal is to get its minimum of $\eta$ on the interval $\big[|A|^2_{\min} , |A|^2_{\max}\big]$.
    
    Taking the derivative of $\eta$, we obtain
    $$\eta'(y)=\frac{1-\frac{2(m^2-3m+3)\epsilon^2 y}{m(m-1)}}{2\sqrt{2\lambda_1\cdot\left(1-\frac{25\epsilon^2}{4}\cdot \frac{m-1}{m}|A|_{\max}^2\right)+y \left(1-\frac{m^2-3m+3}{m(m-1)}\epsilon^2 y\right)}} -\sqrt{\frac{m+1}{m}} \cdot\epsilon.$$
    Then, $\eta''(y)<0$, which implies that $\eta'$ is strictly decreasing on $\big[|A|^2_{\min} , |A|^2_{\max}\big]$. Therefore, 
\begin{align}
&\eta'(y)\geq \eta'(|A|_{\max}^2 )\notag\\
=&\frac{1-\frac{2(m^2-3m+3)\epsilon^2 |A|^2_{\max}}{m(m-1)}}{2\sqrt{2\lambda_1\cdot\left(1-\frac{25\epsilon^2}{4}\cdot \frac{m-1}{m}|A|_{\max}^2\right)+|A|^2_{\max} \left(1-\frac{m^2-3m+3}{m(m-1)}\epsilon^2 |A|^2_{\max}\right)}} -\sqrt{\frac{m+1}{m}} \cdot\epsilon\notag \\
>& \frac{1-\frac{2(m^2-3m+3)\epsilon^2 |A|^2_{\max} }{m(m-1)} }{2|A|_{\max}\sqrt{3- \frac{27m^2-56m+31}{2m(m-1)}|A|^2_{\max} \epsilon^2 }  }  -\sqrt{\frac{m+1}{m}}\cdot \epsilon\notag,
\end{align}
where we use $$\lambda_1 \leq m <|A|^2_{\max}.$$
We can see that the term on the right-hand side of the inequality above is greater than 0 is equivalent to 
\begin{align}
    &\left(1-\frac{2(m^2-3m+3)\epsilon^2 |A|^2_{\max}}{m(m-1)}\right)^2\notag\\
    \geq & \frac{4(m+1)}{m}|A|^2_{\max}\epsilon^2 \left(3- \frac{27m^2-56m+31}{2m(m-1)}|A|^2_{\max} \epsilon^2 \right),\notag
    \end{align}
that is,
\begin{align}
  \frac{2(28m^4-62m^3+19m^2+48m-22)}{m^2(m-1)^2}|A|_{\max}^4 \epsilon^4-\frac{4(4m-3)}{m-1}|A|^2_{\max}\epsilon^2+1\geq 0.\notag
\end{align}
 Under the range of $\epsilon$ \eqref{new restriction  epsion}: $$0<\epsilon \leq \frac{1}{2}\sqrt{\frac{m-1}{4m-3}}\cdot \frac{1}{|A|_{\max}},$$ the inequality above obviously holds, in which $\eta'(y)>0$ and 
\begin{align}
   & \eta(y) \geq \eta (|A|^2_{\min})\notag\\
    =&\sqrt{2\lambda_1\left(1-\frac{25\epsilon^2}{4} \cdot\frac{m-1}{m}|A|_{\max}^2\right)+|A|_{\min}^2 \left(1-\frac{m^2-3m+3}{m(m-1)}\epsilon^2 |A|_{\min}^2\right)}\notag\\
    &-\sqrt{\frac{m+1}{m}} \cdot\epsilon |A|_{\min}^2\notag\\
    =&\sqrt{2\lambda_1+|A|_{\min}^2-\left(\frac{25(m-1)\lambda_1}{2m}|A|^2_{\max}+\frac{m^2-3m+3}{m(m-1)}|A|_{\min}^4 \right)\epsilon^2}\notag\\
     &-\sqrt{\frac{m+1}{m}} \cdot\epsilon |A|_{\min}^2\notag.
    \end{align}

Note that by \eqref{basic integral u} and \eqref{basic integral v},
\begin{align}
\int_{\Sigma} \Big|f \big(\langle \nabla u,\mathbf{n}\rangle -\langle \nabla v,\mathbf{n} \rangle \big)  \Big|\geq \int_{\Omega_1}|\nabla u|^2 +\int_{\Omega_2}|\nabla v|^2.\notag
\end{align}
Substituting this result and the minimum of $\eta$ into Lemma \ref{final formula 5} and then dividing both sides of the inequality by $\int_{\Omega_1}|\nabla u|^2 +\int_{\Omega_2}|\nabla v|^2$ , we get this lemma.
\end{proof}

 Based on Lemma \ref{final formula 6}, we can see that when we choose $\epsilon>0$ to be a small constant, $\lambda_1$ has a lower bound greater than $\frac{m}{2}$. Now, we prove Theorem \ref{main}.
\begin{proof}[Proof of Theorem \ref{main}]
We take $$\epsilon=\frac{1}{2}\sqrt{\frac{m-1}{4m-3}}\cdot \frac{1}{|A|_{\max}}$$
in Lemma \ref{final formula 6}. This gives
\begin{align}
& 6\sqrt{\frac{m}{m-1}}\cdot(2\lambda_1-m) |A|_{\max} + \frac{1}{4\sqrt{m}}(2\lambda_1 -m)^2 +\frac{(m+2)\sqrt{m}}{4}
+\frac{m+1}{2m}\lambda_1\cdot|A|_{\min}\notag\\
>&(2\lambda_1-m)\left( 2\sqrt{\frac{m}{m-1}}\cdot|A|_{\max}+\frac{1}{\epsilon} \right)\notag\\
&+2\epsilon \Bigg[2\lambda_1^2+\left(\frac{m+1}{m}|A|^2_{\min}-2 m\right)\lambda_1 + m^2+m\Bigg] \notag\\
\geq& 2\sqrt{\frac{m+1}{m}}\cdot\lambda_1\sqrt{2\lambda_1+|A|_{\min}^2-\left(\frac{25(m-1)\lambda_1}{2m}|A|^2_{\max}+\frac{m^2-3m+3}{m(m-1)}|A|_{\min}^4 \right)\epsilon^2}\notag\\
>&\frac{\sqrt{3}}{2}\sqrt{\frac{m+1}{m}}\cdot\lambda_1\sqrt{\frac{13}{2}\lambda_1+ 5|A|^2_{\min} }
\notag\\
>&\frac{\sqrt{3}}{2}\sqrt{\frac{m+1}{m}}\cdot\lambda_1\sqrt{\frac{13m}{4}+5|A|_{\min}^2   }.\label{main estimate 1}
\end{align}   
where  the first and second inequalities  are the conclusion of Lemma \ref{final formula 6}; in the first  and third  inequalities, we use
$$\frac{1}{4}\sqrt{\frac{m-1}{m}}\cdot \frac{1}{|A|_{\max}}<\epsilon=\frac{1}{2}\sqrt{\frac{m-1}{4m-3}}\cdot \frac{1}{|A|_{\max}}<\frac{1}{4|A|_{\max}}$$ and 
 $$ |A|_{\max} \geq |A|_{\min},|A|_{\max}>\sqrt{m};$$
in the last step, we use the fact that   $$\lambda_1>\frac{m}{2}.$$

By  $$\frac{m}{2}<\lambda_1 \leq m ,$$ we have 
\begin{align}
&\frac{\sqrt{3}}{2}\sqrt{\frac{m+1}{m}}\cdot\lambda_1 \sqrt{\frac{13m}{4}+5|A|_{\min}^2   }-\frac{m+1}{2m}\lambda_1 |A|_{\min} -\frac{1}{4\sqrt{m}}(2\lambda_1 -m)^2\notag\\
=& \frac{\sqrt{3m(m+1)}}{4}\sqrt{\frac{13m}{4}+5|A|_{\min}^2   }-\frac{m+1}{4}|A|_{\min}\notag\\
&+\left(\frac{\sqrt{3}}{2}\sqrt{\frac{m+1}{m}}\sqrt{\frac{13m}{4}+5|A|_{\min}^2   }- \frac{m+1}{2m}|A|_{\min}  -\frac{\lambda_1-\frac{m}{2}}{\sqrt{m} }\right)(\lambda_1-\frac{m}{2})\notag
\end{align}
we continue from the previous page:
\begin{align}
\geq &\frac{\sqrt{3 m(m+1)}}{4}\sqrt{\frac{13m}{4}+5|A|_{\min}^2   }-\frac{m+1}{4}|A|_{\min}\notag\\
&+ \left(\frac{\sqrt{3}}{2}\sqrt{\frac{m+1}{m}}\sqrt{\frac{13m}{4}+5|A|_{\min}^2   }- \frac{m+1}{2m}|A|_{\min}  -\frac{\sqrt{m}}{2}\right)(\lambda_1-\frac{m}{2})\notag\\
>& \frac{\sqrt{3m(m+1)}}{4}\sqrt{\frac{13m}{4}+5|A|_{\min}^2   }-\frac{m+1}{4}|A|_{\min}\label{mian estimate 3},
\end{align}
where  the last step is because the  term $$\frac{\sqrt{3}}{2}\sqrt{\frac{m+1}{m}}\sqrt{\frac{13m}{4}+5|A|_{\min}^2   }- \frac{m+1}{2m}|A|_{\min}  -\frac{\sqrt{m}}{2}$$is greater than zero by direct computation.

Combining  \eqref{main estimate 1}  and \eqref{mian estimate 3} yields
\begin{align}
 \lambda_1 >\frac{m}{2} + \frac{\sqrt{m^2-1}}{48|A|_{\max}} \left( \sqrt{\frac{39m}{4}+15|A|^2_{\min}  }-\sqrt{\frac{m+1}{m}}|A|_{\min}-\frac{m+2}{\sqrt{m+1}}\right).\label{main estimate 2}
\end{align}
Using $2ab\leq a^2+b^2$ gives
\begin{align}
 (|A|_{\min}+\frac{13}{20}\sqrt{m})^2 =&|A|_{\min} ^2+ \frac{13}{10}\sqrt{m}|A|_{\min}  +\frac{149}{400} m \notag\\
 \leq & \left(1+\frac{13}{20}\right)|A|_{\min}^2+ \left(\frac{13}{20}+\frac{169}{400}\right)m\notag\\
 =& \frac{11}{100}\left(15|A|_{\min}^2+ \frac{39}{4}m \right),\notag
\end{align}
which implies that 
\begin{align}
\lambda_1 >\frac{m}{2} + \frac{\sqrt{m^2-1}}{48|A|_{\max}} \left[\left(\frac{10}{\sqrt{11}}-\sqrt{\frac{m+1}{m}}\right)|A|_{\min}+\frac{13}{2\sqrt{11}}\sqrt{m}-\frac{m+2}{\sqrt{m+1}}\right].\label{main estimate 4 }
\end{align}
This completes the proof of the main theorem.

\end{proof}
\section{Proof of Theorem \ref{curvature upper bound} }\label{curvature-steklov}
In this section, we prove that if the norm square of  the second fundamental form $|A|^2$ is constant, then $|A|^2 $ has an upper bound depending only on dimension $m$ and the area of $\Sigma$. The main idea   is to first choose the appropriate $\epsilon$ and $\beta$ in Lemma \ref{final formula 2} from the previous section to obtain Lemma \ref{upper bound function},  then use the lower bound of the first Steklov eigenvalue to get the following Theorem \ref{steklov'}, and finally combine these two results yields the proof of this conclusion. Here,  we  continue to  use the  notation from the previous section:  $\mathbf{n}$ is the inward-pointing unit normal vector field with respect to $\Omega_1$, $A$ is the corresponding second fundamental form, and the case of $\Omega_2$ differs from that $\Omega_1$ only by a sign; $f$ denotes the eigenfunction with  eigenvalue $\lambda_1$, and $u$ and $v$ denote the harmonic extensions of $f$  to the interior of $\Omega_1$ and $\Omega_2$, respectively. Throughout the proof, we always assume that $|A|^2 >m$.

First, choose $\epsilon$ and $\beta$ in  Lemma \ref{final formula 2} and it yields 
\begin{lem}\label{upper bound function}
If $|A|^2$ is constant, then
\begin{align}
   \frac{\int_{\Omega_1}|\nabla u |^2 +\int_{\Omega_2}|\nabla v|^2  }{\int_{\Sigma} f^2 }<\frac{E(m)}{|A|},\notag
\end{align}
where $$E(m)= \frac{ 9\sqrt{m(m-1)}+\frac{8(m+1)\sqrt{m}}{5\sqrt{m-1 }} }{1-\frac{4(2m^2-3m+2)}{25(m-1)^2}}.$$
\end{lem}
\begin{proof}
  We take  $$\epsilon=\frac{2}{5}\sqrt{\frac{m}{m-1}}\cdot\frac{1}{|A|}, \ \beta=\frac{m+1}{2m}\cdot\epsilon |A|^2$$ in Lemma \eqref{final formula 2}. This gives
\begin{align}
&\left[ \frac{9|A|}{2}\sqrt{\frac{m-1}{m}}(2\lambda_1-m)+\frac{4}{5|A|}\sqrt{\frac{m}{m-1}}(2\lambda_1^2-2m\lambda_1 + m^2+m)\right] \notag\\
&\quad \times \left(\int_{\Omega_1}|\nabla u|^2+ \int_{\Omega_2}|\nabla v|^2 \right)\notag\\
&>  \frac{1}{2}\left(1-\frac{4}{25}\cdot \frac{2m^2-3m+2}{(m-1)^2}\right)|A|^2\cdot \bigintsss_{\Sigma}\big( \langle \nabla v ,\mathbf{n}  \rangle-\langle \nabla u,\mathbf{n} \rangle  \big)^2\notag.
\end{align}
 By the Cauchy-Schwarz inequality and \eqref{basic integral u},\eqref{basic integral v}, we have
\begin{align}
  \int_{\Sigma}\big( \langle \nabla v ,\mathbf{n}  \rangle-\langle \nabla u,\mathbf{n} \rangle  \big)^2 \geq \frac{\Big(\int_{\Sigma}f\big( \langle \nabla v ,\mathbf{n}  \rangle-\langle \nabla u,\mathbf{n} \rangle  \big) \Big)^2  }{\int_{\Sigma} f^2}=\frac{\Big(\int_{\Omega_1}|\nabla u |^2 +\int_{\Omega_2}|\nabla v|^2\Big)^2}{\int_{\Sigma} f^2}\notag.
  \end{align}
By combining the  inequalities above, noting that $|A|^2>m,\lambda_1\leq m$ and  then dividing both sides of the inequality by $\int_{\Omega_1}|\nabla u|^2 +\int_{\Omega_2}|\nabla v|^2$, we complete the proof of this lemma.
\end{proof}
To estimate the lower bound of $$\frac{\int_{\Omega_1}|\nabla u |^2 +\int_{\Omega_2}|\nabla v|^2  }{\int_{\Sigma} f^2 },$$ we need Theorem. For the reader's convenience, we briefly restate the content of Theorem.
\begin{thm}\label{steklov'}
 The first nonzero Steklov eigenvalue satisfies $$ \tau_1 \geq\frac{\textup{Volume}(S^m) }{D(m)\textup{Volume}(\Sigma)}, $$    i.e., 
$$ \tau_1= \inf\limits_{h\in C^1 (\Sigma), \ \int_{\Sigma} h =0 } \frac{\int_{\Omega_1}|\nabla (\mathcal{H}_1 h)|^2+\int_{\Omega_2}|\nabla (\mathcal{H}_2 h)|^2}{\int_{\Sigma} h^2 }\geq \frac{\textup{Volume}(S^m) }{D(m)\textup{Volume}(\Sigma)} , $$ where $\mathcal{H}_1 h $ and $\mathcal{H}_2 h$ denote the harmonic extensions of $h$ to the interior of $\Omega_1$ and $\Omega_2$, respectively; $$D(m)=\frac{1}{\sin(\delta_m)}+(m+1)\delta_m \ ,\ \sin^2(\delta_m)= \frac{2}{\sqrt{4(m+1)^2+1}+1}, \ 0<\delta_m<\frac{\pi}{2}.$$
\end{thm}
\begin{proof}[Proof of Theorem \ref{steklov}]
 Let $g$ be the eigenfunction with eigenvalue with eigenvalue $\rho_1$. Hence ,
\begin{align}
   \int_{\Sigma} g=0, \quad  -\langle\nabla \mathcal{H}_1 g, \mathbf{n} \rangle+ \langle \nabla\mathcal{H}_1 g, \mathbf{n}\rangle =\tau_1 g \label{steklov 1}
   \end{align}
and \begin{align}
    \int_{\Omega_1}|\nabla (\mathcal{H}_1 g)|^2+\int_{\Omega_2}|\nabla (\mathcal{H}_2 g)|^2=\tau_1 \int_{\Sigma} g^2.\label{steklov 2}
    \end{align}
Note that $\mathbf{n}$ is the inward-pointing unit normal vector field with respect to $\Omega_1$ here.
Let $$w =\begin{cases}\mathcal{H}_1 g  , \ \textup{on } \Omega_1 \\\mathcal{H}_2 g  ,\ \textup{on } \Omega_2.\end{cases}$$
Hence, $w$ is a global Lipschitz  function on $S^{m+1}$. 
Let $$\overline{w}=w-\frac{1}{\textup{Volume}(S^{m+1})}\int_{S^{m+1}}w.$$
 By the fact that the first nonzero eigenvalue of $S^{m+1}$ is $m+1$ and \eqref{steklov 2}, we have
 \begin{align}
    \int_{S^{m+1}} \overline{w}^2 \leq \frac{1}{m+1}\int_{S^{m+1}} |\nabla \overline{w}|^2=\frac{\tau_1}{m+1} \int_{\Sigma} g^2.\label{sphere neumann}
 \end{align}
In  what follows, we will prove the mean value formula  for $\overline{w}^2 $ on $S^{m+1}$. 
Moreover, unlike standard proof, since our case is not globally smooth, we must also handle the integral over the boundary $\Sigma$ during the process.

First, denote by $\rho(,)$  the distance in $S^{m+1} $. Fix $y \in\Sigma$ and  let $$B_s (y)=\{z\in S^{m+1}| \rho (y,z) \leq s  \}.$$  In the following  discussion, for convenience, we  use $\rho$ to denote the distance to $y$ before \eqref{mean value 7}.
Then, using the fact that  $$\Delta \cos \rho =-(m+1) \cos \rho $$ on $S^{m+1}$ gives
\begin{align}
  \textup{div} \left( \frac{\nabla \cos \rho}{ \sin^{m+1} \rho    }    \right) =\frac{-\Delta \cos \rho   }{ \sin^{m+1} \rho}+ \frac{(m+1)\langle\nabla \sin \rho, \nabla \cos \rho\rangle}{\sin^{m+2}\rho}=0.\notag
\end{align}
  Then for any $\epsilon>0$ and any $\epsilon<s< \frac{\pi}{2}$, we multiply both sides of the equality above by $\overline{w}^2$ and then apply the divergence theorem  to $\overline{w}^2 \textup{div} \left( \frac{\nabla \cos \rho}{ \sin^{m+1} \rho  }    \right)$  on  $B_s (y)\setminus B_{\epsilon}(y)$ (the divergence theorem holds for Lipschitz  functions ). This yields
\begin{align}
 &\frac{1}{\sin^{m +1}s}  \int_{\partial B_s (y) } \overline{w}^2 \langle\nabla \cos \rho, \nabla \rho \rangle-\frac{1}{\sin^{m+1} \epsilon} \int_{\partial B_\epsilon (y) } \overline{w}^2\langle\nabla \cos \rho, \nabla \rho \rangle\notag\\
 &-\int_{B_s (y)\setminus B_\epsilon (y) } \frac{\langle\nabla \overline{w}^2, \nabla \cos \rho \rangle}{\sin^{m+1} \rho}= 0\ .\label{mean value 1}
\end{align}
Applying the divergence theorem again, we have 
\begin{align}
&\frac{1}{\sin^{m +1}s}  \int_{\partial B_s (y) } \overline{w}^2 \langle\nabla \cos \rho, \nabla \rho \rangle\notag\\
=&\frac{1}{\sin^{m +1}s}  \int_{ B_s (y) } \overline{w}^2 \Delta \cos\rho  + \frac{1}{\sin^{m +1}s}  \int_{ B_s (y) }\langle \nabla \overline{w}^2, \nabla \cos \rho \rangle\notag\\
=& -\frac{m+1}{\sin^{m +1}s}  \int_{ B_s (y) } \overline{w}^2  \cos\rho +  \frac{1}{\sin^{m +1}s}  \int_{ B_s (y) }\langle \nabla \overline{w}^2, \nabla \cos \rho \rangle \label{mean value 2}   
\end{align}
and 
\begin{align}
&\frac{1}{\sin^{m +1}\epsilon}  \int_{\partial B_\epsilon (y) } \overline{w}^2 \langle\nabla \cos \rho, \nabla \rho \rangle\notag\\
=&\frac{1}{\sin^{m +1}\epsilon}  \int_{ B_\epsilon (y) } \overline{w}^2\Delta \cos\rho  + \frac{1}{\sin^{m +1}\epsilon}  \int_{ B_\epsilon (y) }\langle \nabla \overline{w}^2, \nabla \cos \rho \rangle\notag\\
=& -\frac{m+1}{\sin^{m +1}\epsilon}  \int_{ B_\epsilon (y) } \overline{w}^2  \cos\rho +  \frac{1}{\sin^{m +1}\epsilon}  \int_{ B_\epsilon (y) }\langle \nabla \overline{w}^2, \nabla \cos \rho \rangle. \label{mean value 2'} 
\end{align}
Also, using integration by parts and the coarea formula gives 
\begin{align}
&\int_{B_s (y)\setminus B_\epsilon (y) } \frac{\langle\nabla \overline{w}^2, \nabla \cos \rho \rangle}{\sin^{m+1} \rho}\notag\\
=& (m+1)\int_{\epsilon}^{s}\frac{\cos t}{\sin^{m+2}t }dt\int_{B_{t}(y)} \langle \nabla \overline{w}^2, \nabla \cos \rho \rangle+\frac{1}{\sin^{m+1} s} \int_{B_s(y)}\langle \nabla \overline{w}^2, \nabla \cos \rho \rangle\notag\\
&-\frac{1}{\sin^{m+1}\epsilon}\int_{B_{\epsilon}(y) }\langle \nabla \overline{w}^2, \nabla \cos \rho \rangle.\label{mean value 3}
\end{align}
Combining \eqref{mean value 1}, \eqref{mean value 2}, \eqref{mean value 2'} and \eqref{mean value 3},  and letting $\epsilon \rightarrow 0$, we get 
\begin{align}
&\textup{Volume}(S^m) \overline{w}^2 (y)\notag\\
=&\frac{m+1}{\sin^{m+1}s}\int_{B_s(y)} \overline{w}^2 \cos \rho + (m+1)\int_{0}^s\frac{\cos t}{\sin^{m+2}t }dt\int_{B_{t}(y)} \langle \nabla \overline{w}^2, \nabla \cos \rho \rangle.\label{mean value 4}
\end{align}
$\forall 0<t<s$, let $$\tilde{\gamma}=\begin{cases} \cos \rho-\cos t, \textup{on } B_t(y),\\ 0, \textup{otherwise}.
\end{cases}$$
Then $\tilde{\gamma}$ is a global Lipschitz function. 

Note that the divergence theorem  holds for Lipschitz functions again  and the definition of $\overline{w}$: $$\overline{w}=\begin{cases}\mathcal{H}_1 g-\frac{1}{\textup{Volume}(S^{m+1})}\int_{S^{m+1}}w, \ \textup{on } \Omega_1 \\\mathcal{H}_2 g-\frac{1}{\textup{Volume}(S^{m+1})}\int_{S^{m+1}}w,\ \textup{on } \Omega_2.\end{cases}$$ Hence, we obtain
\begin{align}
   & \int_{B_t(y)}  \langle \nabla \overline{w}^2, \nabla \cos \rho \rangle  \notag\\
  =&\int_{\Omega_1} \langle \nabla \overline{w}^2, \nabla \tilde{\gamma}\rangle +\int_{\Omega_2} \langle \nabla \overline{w}^2, \nabla \tilde{\gamma}\rangle  \notag\\
  =&-2\int_{\Sigma}\tilde{\gamma} \overline{w}\langle \nabla (\mathcal{H}_1 g)  ,\mathbf{n} \rangle-\int_{\Omega_1}\tilde{\gamma} \Delta \overline{w}^2+2\int_{\Sigma}\tilde{\gamma}\overline{w} \langle \nabla(\mathcal{H}_2 g),\mathbf{n} \rangle-\int_{\Omega_2}\tilde{\gamma} \Delta \overline{w}^2 \notag\\
  =& 2\tau_1 \int_{B_t(y)\cap\Sigma}(\cos \rho-\cos t)\overline{w} g- 2\int_{B_t(y)}(\cos \rho-\cos t) |\nabla(\mathcal{H}_1g)|^2\notag\\
  &-2\int_{B_t (y)}(\cos \rho-\cos t)|\nabla(\mathcal{H}_2 g)|^2.\label{mean value 5}
  \end{align}
where in the last step, we use \eqref{steklov 2}, the definition of $\tilde{\gamma}$,  and the harmonicity  of  $ \mathcal{H}_1g$ and $\mathcal{H}_2 g$.

Combining \eqref{mean value 4} and \eqref{mean value 5} gives
\begin{align}
&\textup{Volume}(S^m) \overline{w}^2 (y)\notag\\
=&\frac{m+1}{\sin^{m+1}s}\int_{B_s(y)} \overline{w}^2 \cos \rho + 2(m+1)\tau_1\int_{0}^s \frac{\cos t}{\sin^{m+2} t} dt\int_{B_t(y)\cap\Sigma}(\cos \rho-\cos t)\overline{w} g\notag\\
&-2(m+1) \int_{0}^s \frac{\cos t}{\sin^{m+2} t} dt \int_{B_t(y)}(\cos \rho-\cos t) |\nabla(\mathcal{H}_1 g)|^2\notag\\
&-2(m+1) \int_{0}^s \frac{\cos t}{\sin^{m+2} t} dt \int_{B_t(y)}(\cos \rho-\cos t) |\nabla(\mathcal{H}_2 g)|^2\notag\\
\leq & \frac{m+1}{\sin^{m+1}s}\int_{B_s(y)} \overline{w}^2 \cos \rho + 2(m+1)\tau_1\int_{0}^s \frac{\cos t}{\sin^{m+2} t} dt\int_{B_t(y)\cap\Sigma}(\cos \rho-\cos t)\overline{w}g.\label{mean value 6}
\end{align}
We integrate both sides of \eqref{mean value 6} over $\Sigma$ and use Fubini's theorem. This yields
\begin{align}
    &\frac{\textup{Volume}(S^m)}{m+1} \int_{\Sigma}\overline{w}^2 \ d\sigma_y  \notag\\
\leq & \frac{1}{\sin^{m+1}s} \int_{\rho(, \Sigma)\leq s }\overline{w}^2 (z) \ dS_z  \int_{B_s (z) \cap\Sigma} \cos (\rho(z,y))\  d\sigma_y\notag\\
&+ 2\tau_1 \int_{0}^s \frac{\cos t}{\sin^{m+2} t} dt \int_{\Sigma } \overline{w}(z) g(z)\  d\sigma_z \int_{B_{t}(z)}\big(\cos (\rho(z,y)-\cos t)\big)\  d\sigma_y\notag\
\end{align}
we continue from the previous page:
\begin{align}
\leq & \frac{1}{\sin^{m+1}s} \int_{\rho(, \Sigma)\leq s }\overline{w}^2 (z) \ dS_z  \int_{B_s (z) \cap\Sigma} \cos (\rho(z,y))\  d\sigma_y\notag\\
&+\tau_1 \int_{0}^s \frac{1}{\sin^{m} t} dt \int_{\Sigma } \overline{w}(z) g(z)\  d\sigma_z \int_{B_{t}(z)}\cos (\rho(z,y)) \  d\sigma_y.\label{mean value 7}
\end{align}
where $dS$ and $d\sigma$ denote the volume element of $S^{m+1}$ and $\Sigma$, respectively; in the last step, we use the the inequality $$\cos a (\cos b-\cos a) \leq \frac{1}{2}\cos b \sin^2 a,\  0\leq b\leq a<\frac{\pi}{2}.  $$

By Proposition \ref{volume growth}, we know that for any  $0<t<\frac{\pi}{2}$ and any $z\in S^{m+1}$, we have
$$ \int_{B_t (z) \cap \Sigma}\cos(\rho(z,y))d\sigma_y \leq \textup{Volume}(\Sigma) \sin^{m} s.  $$
We substitute this result and \eqref{sphere neumann} into \eqref{mean value 7} to obtain 
\begin{align}
&\frac{\textup{Volume}(S^m)}{m+1} \int_{\Sigma}\overline{w}^2 \ d\sigma_y  \notag\\
\leq & \frac{\textup{Area}(\Sigma)}{\sin s} \int_{S^{m+1} }\overline{w}^2 (z) \ dS_z + \tau_1 \textup{Volume}(\Sigma) s\int_{\Sigma} \overline{w}(z)g(z) d\sigma_z\notag\\
\leq & \tau_1\textup{Volume}(\Sigma) \left(\frac{1}{(m+1) \sin s} \int_{\Sigma}g^2 +s\int_{\Sigma} \overline{w}(z)g(z) d\sigma_z\right).\label{mean value 8}
\end{align}
From the definition of $\overline{w}$ and \eqref{steklov 1}, we find that 
$$\int_{\Sigma}g^2 \leq \int_{\Sigma} \overline{w} ^2. $$
Using this inequality and the Cauchy- Schwarz inequality  yields
\begin{align}
 &\frac{\textup{Volume}(S^m)}{m+1} \int_{\Sigma}\overline{w}^2  \notag\\
\leq &\tau_1\textup{Volume}(\Sigma) \left(\frac{1}{(m+1) \sin s} \int_{\Sigma}g^2 +s\sqrt{\int_{\Sigma} \overline{w}^2 } \sqrt{\int_{\Sigma} g^2 }\right)\notag\\
\leq & \tau_1\textup{Volume}(\Sigma) \left(\frac{1}{(m+1) \sin s} \int_{\Sigma}\overline{w}^2 +s\int_{\Sigma}\overline{w}^2\right),\notag
\end{align}
i.e.,
\begin{align}
    \frac{\textup{Volume}(S^m)}{m+1} \int_{\Sigma}\overline{w}^2\leq \tau_1\textup{Volume}(\Sigma)\left( \frac{1}{(m+1)\sin s} +s \right),\label{mean value 9}
    \end{align}
which holds for any $0<s<\frac{\pi}{2}$.

Let $ \delta_{m}$ be a constant  satisfying $$\sin^2(\delta_m)= \frac{2}{\sqrt{4(m+1)^2+1}+1}\ , \ 0<\delta_m<\frac{\pi}{2}.$$ The right-hand of the inequality attains its maximum at $\delta_m $.
Then the theorem follows by taking $s=\delta_m$.
\end{proof}
Now, we prove Theorem \ref{upper bound function}.
\begin{proof}[Proof of Theorem \ref{upper bound function}]
  From Theorem \ref{steklov}, we have
  \begin{align}
\frac{\int_{\Omega_1}|\nabla u |^2 +\int_{\Omega_2}|\nabla v|^2  }{\int_{\Sigma} f^2 }\geq \tau_1\geq \frac{\textup{Volume}(S^m) }{D(m)\textup{Volume}(\Sigma)}.\label{final steklov}
\end{align}
Then the theorem follows from Lemma \ref{upper bound function} and \eqref{final steklov}.

\end{proof}

 \vspace{0.5cm}
 \textbf{\Large Declarations }

\vspace{0.2cm}
\textbf{Conflict and interest:} The author states that there is no conflict of interest.

	\bibliographystyle{siam}
	\bibliography{ref}
 \end{document}